\documentclass[11pt,oneside,reqno]{amsart}

\usepackage{amsmath,amssymb,amsthm,textcomp,mathtools}
\usepackage{amsfonts,graphicx}
\usepackage[mathscr]{eucal}
\usepackage{color}
\usepackage{csquotes}
\usepackage{enumitem}
\numberwithin{equation}{section}

\theoremstyle{definition}

\numberwithin{equation}{section}

\newtheorem{theorem}{\bf Theorem}[section]
\newtheorem{remark}{\bf Remark}[section]
\newtheorem{proposition}{Proposition}[section]
\newtheorem{lemma}{Lemma}[section]

\newtheoremstyle
{remarkstyle}
{}
{11pt}
{}
{}
{\bfseries}
{:}
{     }
{\thmname{#1} \thmnumber{#2} }

\theoremstyle{remarkstyle}

\begin{document}
	\title{Fractional Poisson Random Fields on $\mathbb{R}^2_+$}
	\author{K. K. Kataria}
	\address{Kuldeep Kumar Kataria, Department of Mathematics,
		 Indian Institute of Technology Bhilai, Durg-491002, India.}
	 \email{kuldeepk@iitbhilai.ac.in}
	 \author[Pradeep Vishwakarma]{P. Vishwakarma}
	 \address{Pradeep Vishwakarma, Department of Mathematics,
	 	Indian Institute of Technology Bhilai, Durg-491002, India.}
	 \email{pradeepv@iitbhilai.ac.in}

	\subjclass[2010]{Primary : 60G55; Secondary: 60G57, 60G60}
	
	\keywords{ fractional Poisson random field, Caputo fractional derivative, inverse stable subordinator, Adomian decomposition method, generalized Wright function, order statistics, generalized fractional Poisson process, random motion}
	\date{\today}
	
	\maketitle
	\begin{abstract}
		In this paper, we consider a fractional Poisson random field (FPRF) on positive plane. It is defined as a process whose one dimensional distribution is the solution of a system of fractional partial differential equations. A time-changed representation for the FPRF is given in terms of the composition of Poisson random field with a bivariate random process. Some integrals of the FPRF are introduced and studied. Using the Adomian decomposition method, a closed form expression for its probability mass function is obtained in terms of the generalized Wright function. Some results related to the order statistics of random numbers of random variables are presented. Also, we introduce a generalization of Poisson random field on $\mathbb{R}^d_+$, $d\ge1$ which reduces to the Poisson random field in a special case. For $d=1$, it further reduces to a generalized Poisson process (GPP). A time-changed representation for the GPP is established. Moreover, we construct a time-changed linear and planer random motions of a particle driven by the GPP. The conditional distribution of the random position of particle is derived.  Later, we define the compound fractional Poisson random field via FPRF. Also, a generalized version of it on $\mathbb{R}^d_+$, $d\ge1$ is discussed.
	\end{abstract}
	
\section{Introduction}
 The Poisson random field (PRF) also known as the spatial Poisson point process on $\mathbb{R}^d$, $d\ge1$ represents the random arrangements of  points in $d$-dimensional space. It is an important statistical models for examining the points pattern, where the points denote the spatial locations of various objects. The PRF as mathematical model has been used in various disciplines of engineering and science such as physics (see Scargle (1998)), astronomy (see Babu and Feigelson (1996)), ecology (see Thompson (1955)) and telecommunication (see Andrews \textit{et al.} (2010)). In physics, it has potential application in the study of an ensemble average of field theories (see Peng (2021)). Also, the PRF on two dimensional Euclidean space can be use to model the intermittent phenomenon of physical system (see Field Jr. and Grigoriu (2011)).  For $d=1$, the PRF reduces to the well known homogeneous Poisson process. 

Over the past three decades, the global memory property of time fractional processes has made them a significant topic of study. Long memory is a very important characteristics of almost every real system and these time fractional processes provide a better mathematical models for such systems.
A fractional Poisson process on the positive real line can be defined in four different ways. The first method adopts the integral representation of the fractional Brownian motion where the Brownian motion is replaced by the Poisson process (see Wang and Wen (2003)). For further details on this method, we refer the reader to Laskin (2003) and Cahoy \textit{et al.} (2010). The second method is referred as the renewal approach in which the fractional Poisson process is defined as a sum of independent non-negative random variables which has Mittag-Leffler distribution (see Mainardi \textit{et al.} (2005)).  In the third method, the integer order derivative involved in the governing system of differential equations of the Poisson process is replaced by the Caputo fractional derivative (see Beghin and Orsingher ((2009), (2010))). The fourth method is referred as the time-changed approach in which the time is replaced by an inverse stable subordinator (see Meerscheart \textit{et al.} (2011)).
Following the approach of Meerschaert {\it et al.} (2011), Leonenko and Merzbach (2015) introduced a fractional variant of the Poisson random field on $\mathbb{R}^2_+$. Further,  Aletti \textit{et al.} (2018) obtained a martingale characterization of the fractional Poisson processes and various characterizations for the FPRF on $\mathbb{R}^2_+$. Moreover, they derived the system of fractional differential equations that governs the distribution of FPRF. Because of their complex form, these differential equations are difficult to solve.

In this paper, we define a FPRF on $\mathbb{R}^2_+$ using a system of fractional differential equations that governs the distribution of FPRF. The system of differential equations obtained here are different from the equations derived by Aletti \textit{et al.} (2018). By using the Adomian decomposition method, we obtain the explicit form of the probability mass function (pmf) of FPRF in terms of the generalized Wright function.

The paper is organized as follows: 

In the next section, we collect some definitions and known results that will be used throughout this paper.

In Section \ref{sec3}, first, we discuss the definition of homogeneous Poisson random field on $d$-dimensional Euclidean space and  recall some known results for $d=2$. We derive the system of partial differential equations that governs the pmf of PRF on $\mathbb{R}^2_+$. In this system of differential equations, we replace the integer order derivatives with Caputo fractional derivatives to obtain the governing system of partial fractional differential equations for the FPRF. The explicit expressions of its mean, variance and covariance are obtained. Moreover, we establish a connection between the FPRF and PRF via a bivariate random process whose components processes are independent and their density functions are the folded solutions of some specific fractional diffusion equations. We observe that the FPRF and the fractional variant of PRF defined by Leonenko and Merzbach (2015) are  equal in distribution.  Also, we introduce some integrals of the PRF and FPRF.  The variance and a conditional variance of the fractional integral of PRF are obtained.

In Section \ref{sec4}, we apply the Adomian decomposition method to solve the system of differential equations that governs the pmf of FPRF. We derive a closed form expression of the distribution of FPRF in terms of the generalized Wright function. A series expression for the capacity function of FPRF is obtained. Also, some results on the distribution functions of order statistics of random numbers of independent and identically distributed random variables are obtained. Here, the random number of random variables is modeled by the FPRF. Also, we study some alternate fractional variants of the PRF. 

In Section \ref{sec5}, we introduce a generalized fractional Poisson random field on $\mathbb{R}^d_+$, $d\ge1$. For $d=2$, it gives another fractional generalization of the PRF on $\mathbb{R}^2_+$. For $d=1$, it reduces to a generalized Poisson process (GPP). It is shown that the GPP is related to the Poisson process via a random process whose Laplace transform can be expressed in terms of the generalized Mittag-Leffler function. That is, it is a time-changed version of the Poisson process. Moreover, we introduce a time changed version of the PRF on $\mathbb{R}^2_+$, where the time space is replaced by a bivariate random process.

In Section \ref{sec6}, we study time-changed linear and planer random motions of a particle moving with finite velocity which changes its direction at every GPP events. In linear case, we derive the Fourier transform of the random position of particle at time $t>0$ and analyze its limiting value in a particular case. Also, we obtain the conditional characteristics function and conditional joint density of the position vector of a particle randomly moving in plane and changing its direction at every GPP arrivals.

In Section \ref{sec7}, we consider the compound fractional Poisson random field. The series representation of its distribution function is obtained. Also, a generalized version of the compound fractional Poisson random field on $\mathrm{R}^d_+$ is discussed. 
\section{Preliminaries}\label{pre}
Here, we collect some known results which will be used in this paper. First, we fix some notations.

 \subsection{Notations} Let $\mathbb{R}$ and $\mathbb{N}$ denote the sets of real numbers and positive integers, respectively. Also, the set of non-negative real numbers is denoted by $\mathbb{R}_+=\mathbb{R}\cup\{0\}$. For $r\in\mathbb{R}$ and $n\in\mathbb{N}$, we denote the falling and rising factorials by $r_{(n)}=r(r-1)\dots(r-n+1)$  and $(r)_n=r(r+1)\dots(r+n-1)$, respectively. 
 \subsection{Some special functions} Here, we recall the definitions of some special functions. 
 \subsubsection{Beta and incomplete beta functions}
 For $\alpha>0$ and $\beta>0$, the incomplete beta function is defined as follows:
 \begin{equation}\label{ibeta}
 	B(x;\alpha,\beta)=\int_{0}^{x}w^{\alpha-1}(1-w)^{\beta-1}\,\mathrm{d}w,\ 0<x\leq1.
 \end{equation}
 In particular, for $x=1$, it reduces to the beta function defined by
 \begin{equation}\label{beta}
 	B(\alpha,\beta)=\int_{0}^{1}w^{\alpha-1}(1-w)^{\beta-1}\,\mathrm{d}w=\frac{\Gamma(\alpha)\Gamma(\beta)}{\Gamma(\alpha+\beta)},
 \end{equation}
 where $\Gamma(\cdot)$ is the gamma function.
 
 Further, on applying the Taylor series expansion, we have
 $
 (1-w)^{\beta-1}=1+(1-\beta)w+O(w^2)\ \ \text{as}\ \ w\to 0.
 $ So, we get the following approximation of the incomplete beta function:
 \begin{equation}\label{incbetalim}
 	B(x;\alpha,\beta)=\frac{x^\alpha}{\alpha}+(1-\beta)\frac{x^{\alpha+1}}{\alpha+1}+O(x^{\alpha+2})\ \ \text{as}\ \ x\to0.
 \end{equation}
 Here, $f(x)=O(g(x))$ for some positive function $g(\cdot)$ implies that there exist $C>0$ such that $|f(x)|\leq Cg(x)$ for all $x$.
 \subsubsection{Generalized Wright function} For $a_i$, $b_j\in\mathbb{R}$ and $\alpha_i$, $\beta_j\in\mathbb{R}-\{0\}$, $i\in\{1,2,\dots,m\}$, $j\in\{1,2,\dots,l\}$, the generalized Wright function is defined by the following series (see Haubold \textit{et al.} (2011)):
 \begin{equation}\label{genwrit}
 	{}_m\Psi_l\left[\begin{matrix}
 		(a_i,\alpha_i)_{1,m}\\\\
 		(b_j,\beta_j)_{1,l}
 	\end{matrix}\Bigg| x \right]={}_m\Psi_l\left[\begin{matrix}
 		(a_1,\alpha_1),\,(a_2,\alpha_2),\dots,(a_m,\alpha_m)\\\\
 		(b_1,\beta_1),\,(b_2,\beta_2),\dots,(b_l,\beta_l)
 	\end{matrix}\Bigg| x \right]=\sum_{n=0}^{\infty}\frac{\prod_{i=1}^{m}\Gamma(a_i+n\alpha_i)x^n}{\prod_{j=1}^{l}\Gamma(b_j+n\beta_j)n!},\ x\in\mathbb{R}.
 \end{equation}
 For $m=0$ and $l=1$, it reduces to the Wright function defined as follows:
 \begin{equation}\label{wright}
 	W_{\beta,b}(x)=\sum_{n=0}^{\infty}\frac{x^n}{\Gamma(n+1)\Gamma(n\beta+b)},\ x\in\mathbb{R}.
 \end{equation}	
 \subsubsection{Airy function} The Airy function is defined by the following integral:
 \begin{equation}\label{airy}
 	\mathcal{A}i(x)=\frac{1}{\pi}\int_{0}^{\infty}\cos\left(wx+\frac{w^3}{3}\right)\,\mathrm{d}w=\frac{\sqrt{x}}{3}\left(I_{-1/3}\left({2x^{3/2}}/{3}\right)-I_{1/3}\left({2x^{3/2}}/{3}\right)\right),\ x\in\mathbb{R},
 \end{equation} 
 where $I_\nu(\cdot)$ is the modified Bessel function of first kind of order $\nu$.
 
 \subsubsection{Generalized Mittag-Leffler function} For $\alpha>0$, $\beta>0$ and $\gamma>0$, the generalized Mittag-Leffler function is defined by the following absolutely convergent series:
 \begin{equation}\label{mittag}
 	E_{\alpha,\beta}^\gamma(x)=\sum_{n=0}^{\infty}\frac{(\gamma)_nx^n}{\Gamma(n\alpha+\beta)n!},\ x\in\mathbb{R},
 \end{equation}
 where $(\gamma)_{n}=\Gamma(\gamma+n)/\Gamma(\gamma)$ is the rising factorial. Its integer order derivative is given by 
 \begin{equation}\label{der2mittag}
 	\frac{\mathrm{d}^k}{\mathrm{d}x^k}E_{\alpha,\beta}^\gamma(x)=(\gamma)_kE^{\gamma+k}_{\alpha,\beta+k\alpha}(x),\ k\ge0.
 \end{equation}
 Also, we have the following result:
 \begin{equation}\label{wetlapmittag}
 	\int_{0}^{\infty}e^{-wt}t^{\beta-1}E^{\gamma}_{\alpha,\beta}(\lambda t^\alpha)\,\mathrm{d}t=\frac{w^{\alpha\gamma-\beta}}{(w^\alpha-\lambda)^\gamma},\ |w|>|\lambda|^{\alpha^{-1}}.
 \end{equation}
 
 For $\gamma=1$, (\ref{mittag}) reduces to the two parameter Mittag-Leffler function. Further, for $\gamma=\beta=1$, it is called the one parameter Mittag-Leffler function given by
 \begin{equation*}
 	E_{\alpha,1}(x)=\sum_{n=0}^{\infty}\frac{x^n}{\Gamma(n\alpha+1)},\ x\in\mathbb{R}.
 \end{equation*}
 Its Laplace transform is
 \begin{equation}\label{MittagLlap}
 	\int_{0}^{\infty}e^{-sx}E_{\alpha,1}(cx^\alpha)\,\mathrm{d}x=\frac{s^{\alpha-1}}{s^\alpha-c},\ c\in\mathbb{R},\,s>0.
 \end{equation}
 Moreover, for $\alpha=1/2$, we have
 \begin{equation*}
 	E_{1/2,1}(x)=e^{x^2}\left(1+\frac{2}{\sqrt{\pi}}\int_{0}^{x}{e^{-s^2}}\,\mathrm{d}s\right),
 \end{equation*}
 which has the following asymptotic behaviour:
 \begin{equation}\label{Mittaglimit}
 	E_{1/2,1}(x)\sim 2e^{x^2}\ \ \text{as}\ |x|\to\infty.
 \end{equation}
\subsection{Fractional integral and derivative} Here, we collect  definitions and results on Riemann-Liouville integral and Caputo fractional derivative (see Kilbas \textit{et al. (2006)}). 
\subsubsection{Riemann-Liouville integral} The Riemann-Liouville fractional integral of order $\nu>0$ of an integrable function $f(\cdot)$ is defined by
\begin{equation}\label{RL}
	I^\nu_t (f(t))=\frac{1}{\Gamma(\nu)}\int_{0}^{t}\frac{f(s)\,\mathrm{d}s}{(t-s)^{1-\nu}},
\end{equation}
where $I^\nu_t$ is called the fractional integral operator of order $\nu$.
\begin{lemma}\label{lemma1}
	For $\alpha>0$, the fractional integral of $t^{\alpha-1}$ is
\end{lemma}
\begin{equation}
	I_t^\nu (t^{\alpha-1})=\frac{\Gamma(\alpha)}{\Gamma(\alpha+\nu)}t^{\alpha+\nu-1},\ t>0.
\end{equation}
\subsubsection{Caputo fractional derivative} The Caputo fractional derivative of order $\nu\in(0,1]$ of a function $f(\cdot)$ is define as follows:
\begin{equation}\label{Caputoder}
	\frac{\mathrm{d}^\nu}{\mathrm{d}t^\nu}f(t)=\begin{cases}
		\frac{1}{\Gamma(1-\nu)}\int_{0}^{t}\frac{\frac{\mathrm{d}}{\mathrm{d}s}f(s)}{(t-s)^\nu}\,\mathrm{d}s,\ 0<\nu<1,\\
		\frac{\mathrm{d}}{\mathrm{d}t}f(t),\ \nu=1.
	\end{cases}
\end{equation}
Its Laplace transform is given by 
\begin{equation}\label{Caputoderlap}
	\int_{0}^{\infty}e^{-st}\frac{\mathrm{d}^\nu}{\mathrm{d}t^\nu}f(t)\,\mathrm{d}t=s^\nu\int_{0}^{\infty}e^{-st}f(t)\,\mathrm{d}t-s^{1-\nu}f(0).
\end{equation}

\subsection{Stable subordinator and its inverse} A subordinator $\{S_\nu(t),\ t\ge0\}$, $\nu\in(0,1)$ is a one dimensional L\'evy process with non-decreasing sample path. Its Laplace transform is 
$
	\mathbb{E}e^{-zS_\nu(t)}=e^{-tz^\nu}.
$

The inverse stable subordinator
$
L_\nu(t)=\mathrm{inf}\{s>0:S_\nu(s)>t\},\ t>0
$
is the time when $S_\nu(\cdot)$ crosses the level $t>0$ for the first time. Let $\psi_{L_\nu}(\cdot,t)$ denote the density function of $L_\nu(t)$.
Then, its Laplace transforms are given as follows (see Meerschaert \textit{et. al.} (2011)):
\begin{equation}\label{subdenlap1}
	\int_{0}^{\infty}e^{-zt}\psi_{L_\nu}(x,t)\,\mathrm{d}t=z^{\nu-1}e^{-xz^\nu},\ z>0
\end{equation}
and
$ \int_{0}^{\infty}e^{-zx}\psi_{L_\nu}(x,t)\,\mathrm{d}x=E_{\nu,1}(-zt^\nu),
$
where $E_{\nu,1}(\cdot)$ is the one parameter Mittag-Leffler function.

\subsection{Adomian decomposition method} Here, we give a brief description of the Adomian decomposition method. 

Let us consider the following functional equation:
\begin{equation}\label{1.1}
	u=\phi+T(u),
\end{equation}
where $T$ is a nonlinear operator and $\phi$ is some known function. It is assumed that the solution $u$ and  the nonlinear term $T(u)$ of (\ref{1.1}) can be expressed in the form of absolutely convergent series $u=\sum_{n=0}^{\infty}u_n$ and $T(u)=\sum_{n=0}^{\infty}A_n(u_0,u_1,\ldots,u_n)$, respectively (see Adomian (1986), (1994)). Here, $A_n$ denotes the $n$th Adomian polynomial in terms of $u_0,u_1,\ldots,u_n$. Thus, (\ref{1.1}) can be rewritten as
\begin{equation*}
\sum_{n=0}^{\infty}u_n=\phi+\sum_{n=0}^{\infty}A_n(u_0,u_1,\dots,u_n).
\end{equation*}
In Adomian decomposition method, the series components $u_n$'s are obtained by the recursive relation
$
	u_0=\phi$ and $u_n=A_{n-1}(u_0,u_1,\dots,u_{n-1}).
$
The computation of the Adomian polynomials is a crucial part of the Adomian decomposition method. Moreover, if $T$ is linear, for example, $T(u)=u$ then $A_n$ simply reduces to $u_n$. For more details on these polynomials, we refer the reader to Rach (1984), Duan ((2010), (2011)), Kataria and Vellaisamy (2016).

\section{Poisson random fields}\label{sec3}
In this section, we introduce and study a fractional variant of the Poisson random field (PRF). First, we briefly discuss the homogeneous PRF and a time-changed version of it, namely, the fractional Poisson field studied by Leonenko and Merzbach (2015). For more details on PRF and its fractional version, we refer the reader to Merzbach and Nualart (1986), Stoyen \textit{et al.} (1995), Aletti \textit{et al.} (2018). 

Let $\mathcal{B}_{\mathbb{R}^d}$ be the Borel sigma algebra on $\mathbb{R}^d$, $d\ge1$, and let $|B|$ denote the Lebesgue measure of set $B\in\mathcal{B}_{\mathbb{R}^d}$. Then, the PRF on $\mathbb{R}^d$ is an integer valued random measure $\{\mathcal{N}(B),\ B\in\mathcal{B}_{\mathbb{R}^d}\}$, where $\mathcal{N}(B)$ denotes the random number of points inside a Borel set $B$. It is also called the spatial Poisson point process. It can be characterized as follows (see Stoyen \textit{et al.} (1995), Baddeley (2007)):

\noindent $(i)$ There exist a $\lambda>0$ such that, for $B\subset{\mathbb{R}^d}$ with $|B|<\infty$, the random points count $N(B)$ in set $B$ has Poisson distribution with mean $\lambda|B|$, that is,
\begin{equation}\label{dprfdist}
	\mathrm{Pr}\{\mathcal{N}(B)=k\}=\frac{e^{-\lambda|B|}(\lambda|B|)^k}{k!},\ k\ge0,
\end{equation}
\noindent $(ii)$ for any finite collection $\{B_1,B_2,\dots,B_m\}$ of disjoint  sets, the random variables $\mathcal{N}(B_1)$, $\mathcal{N}(B_2),\dots,\mathcal{N}(B_m)$ are independent  of each other.

Here, $\lambda$ is called the parameter of PRF. The void probability of PRF is $\mathrm{Pr}\{\mathcal{N}(B)=0\}=e^{-\lambda|B|}$. Also, the contact distribution function of PRF is given by
$
	H_B(a)=1-\mathrm{Pr}\{\mathcal{N}(aB)=0\},\ a>0,
$
where $B$ is a Borel set with positive measure such that $a_1B\subset a_2B$ whenever $a_1< a_2$.

 Equivalently, it can be defined using the infinitesimal probabilities as follows (see Stoyen \textit{et al.} (1995)): 
 
 Let $B$ be a Borel set with infinitesimal Lebesgue measure, that is, $o(|B|)/|B|\to0$ as $|B|\to0$. Then, the probability of no point in $B$ is $\mathrm{Pr}\{\mathcal{N}(B)=0\}=1-\lambda|B|+o(|B|)$, probability of exactly one point in $B$ is $\mathrm{Pr}\{\mathcal{N}(B)=1\}=\lambda|B|+o(|B|)$ and the probability of more than one points in $B$ is $\mathrm{Pr}\{\mathcal{N}(B)\ge1\}=o(|B|)$.

 \subsection{Poisson random field on $\mathbb{R}^2_+$} From this point onward, we shall restrict our self to the case of $d=2$. Let us consider the PRF $\{\mathcal{N}(t_1,t_2), (t_1,t_2)\in\mathbb{R}^2_+\}$ on $\mathbb{R}^2_+$ which is defined as
$
 	\mathcal{N}(t_1,t_2)\coloneqq\mathcal{N}([0,t_1]\times[0,t_2])$, $ (t_1,t_2)\in\mathbb{R}_+^2.
$
 Here, $\mathcal{N}(t_1,t_2)$ denotes the total number of random points inside the rectangle $[0,t_1]\times[0,t_2]$. 
 The following properties hold for PRF on $\mathbb{R}^2_+$  (see Merzbach and Nualart (1986)):
 
 Let $(\Omega,\mathcal{E},P)$ be a probability space and $\{\mathcal{E}_{t_1,t_2},(t_1,t_2)\in\mathbb{R}^2_+\}$ be a collection of sub-sigma algebras of $\mathcal{E}$ such that \\
 \noindent $(i)$ $\mathcal{E}_{0,0}$ contains all the null sets in $\mathcal{E}$,\\
 \noindent $(ii)$ $\mathcal{E}_{\tau_1,\tau_2}\subseteq \mathcal{E}_{t_1,t_2}$ whenever $\tau_1\leq t_1$ and $\tau_2\leq t_2$,\\
 \noindent $(iii)$ for all $r\in\mathbb{R}^2_+$, $\mathcal{E}_r=\cap_{r\prec r'}\mathcal{E}_{r'}$. Here, $r=(\tau_1,\tau_2)\prec r'=(t_1,t_2)$ denotes the partial ordering on $\mathbb{R}^2_+$, that is, $r\prec r'$ implies $\tau_1\leq t_1$ and $\tau_2\leq t_2$.
 
For $(\tau_1,\tau_2)\prec (t_1,t_2)$, the increment of PRF on rectangle $(\tau_1,t_1]\times (\tau_2,t_2]$ is given by
\begin{equation*}
	\mathcal{N}((\tau_1,t_1]\times(\tau_2,t_2])=\mathcal{N}(t_1,t_2)-\mathcal{N}(\tau_1,t_2)-\mathcal{N}(t_1,\tau_2)+\mathcal{N}(\tau_1,\tau_2).
\end{equation*}
 The PRF $\{\mathcal{N}(t_1,t_2),\ (t_1,t_2)\in\mathbb{R}^2_+\}$ on $\mathbb{R}^2_+$ is a c\`adl\`ag field adapted to $\{\mathcal{E}_{t_1,t_2}, (t_1,t_2)\in\mathbb{R}^2_+\}$ such that $\mathcal{N}(0,t_2)=\mathcal{N}(t_1,0)=0$ with probability one. The increment $\mathcal{N}((\tau_1,t_1]\times(\tau_1,t_2])$ is a Poisson random variable with mean $\lambda(t_1-\tau_1)(t_2-\tau_2)$. Moreover, on taking $\tau_1=\tau_2=0$, we get the probability mass function (pmf) $p(k,t_1,t_2)=\mathrm{Pr}(\mathcal{N}(t_1,t_2)=k)$ of PRF as follows:
 \begin{equation}\label{prfdist}
 	p(k,t_1,t_2)=\frac{e^{-\lambda t_1t_2}(\lambda t_1t_2)^k}{k!},\ k=0,1,2,\dots.
 \end{equation}
On taking the derivatives on both sides of (\ref{prfdist}) with respect to $t_2$ and $t_1$, successively, we get the following system of partial differential equations that governs the pmf of PRF:
 	\begin{equation}\label{pmfequ1}
 		\frac{\partial^2}{\partial t_1\partial t_2}p(k,t_1,t_2)=\lambda(k+1)p(k+1,t_1,t_2)-\lambda(2k+1) p(k,t_1,t_2)+\lambda k p(k-1,t_1,t_2),\ k\ge0,
 	\end{equation} 
 	with initial condition $p(0,0,0)=1$. Here, $p(-1,t_1,t_2)=0$ for all $(t_1,t_2)\in\mathbb{R}^2_+$.

 \begin{remark}
 	From (\ref{pmfequ1}), the differential equation that governs the probability generating function (pgf) $\mathcal{G}(z,t_1,t_2)=\mathbb{E}z^{\mathcal{N}(t_1,t_2)}$, $|z|\leq1$ of PRF is 
 	\begin{equation}\label{pgf}
 		\frac{\partial^2}{\partial t_1\partial t_2}\mathcal{G}(z,t_1,t_2)=\lambda(z-1)^2\frac{\partial}{\partial z}\mathcal{G}(z,t_1,t_2)+\lambda(z-1)\mathcal{G}(z,t_1,t_2)
 	\end{equation}
 	with $\mathcal{G}(z,0,t_2)=\mathcal{G}(z,t_1,0)=1$ for all $t_1\ge0$ and $t_2\ge0$.
 \end{remark} 
\subsection{Fractional Poisson random field}
Let $\{L_{\nu_1}(t)\}_{t\ge0}$, $0<\nu_1<1$ and  $\{L_{\nu_1}(t)\}_{t\ge0}$, $0<\nu_2<1$ be two independent inverse stable subordinators, and let $\{\mathcal{N}(t_1,t_2),\ (t_1,t_2)\in\mathbb{R}_+^2\}$ be the PRF which is independent of $L_{\nu_1}(t_1)$ and $L_{\nu_2}(t_2)$. Leonenko and Merzbach (2015)  introduced and studied a fractional variant of  the PRF on $\mathbb{R}^2_+$, namely, the fractional Poisson field. It is defined as follows:
\begin{equation}\label{def1}
	\mathcal{N}_{\nu_1,\nu_2}(t_1,t_2)=\mathcal{N}(L_{\nu_1}(t_1),L_{\nu_2}(t_2)),\ (t_1,t_2)\in\mathbb{R}_+^2.
\end{equation}
Recently, Aletti \textit{et al.} (2018) studied this process and obtained various characterizations for it.	Its pmf $p_{\nu_1,\nu_2}(k,t_1,t_2)=\mathrm{Pr}\{\mathcal{N}_{\nu_1,\nu_2}(t_1,t_2)=k\}$, $k\ge0$ is given by 
	\begin{equation*}
		p_{\nu_1,\nu_2}(k,t_1,t_2)=\frac{\lambda^k}{k!\nu_1\nu_2}\int_{0}^{\infty}\int_{0}^{\infty}e^{-\lambda s_1s_2}s_1^{k-1}s_2^{k-1}W_{-\nu_1,0}(-s_1/t_1^{\nu_1})W_{-\nu_2,0}(-s_2/t_2^{\nu_2})\,\mathrm{d}s_1\,\mathrm{d}s_2,
	\end{equation*}
	where $W_{\mu,\nu}(\cdot)$ is the Wright function defined in (\ref{wright}).

	Here, we define a fractional variant of  the PRF using a system of fractional partial differential equations. We call this fractional variant as the fractional Poisson random field (FPRF) and denote it by $\{\mathscr{N}_{\nu_1,\nu_2}(t_1,t_2),(t_1,t_2)\in\mathbb{R}^2_+\}$, $0<\nu_1\leq1,\,0<\nu_2\leq1$.  We show that it is equal in distribution to the process defined in (\ref{def1}). It is a random field whose pmf $q_{\nu_1,\nu_2}(k,t_1,t_2)=\mathrm{Pr}\{\mathscr{N}_{\nu_1,\nu_2}(t_1,t_2)=k\}$ solves the following system of fractional partial differential equations:
	\begin{align}\label{pmfequ2}
	\frac{\partial^{\nu_1+\nu_2}}{\partial t_1^{\nu_1}\partial t_2^{\nu_2}}q_{\nu_1,\nu_2}(t_1,t_2)&=\lambda(k+1)q_{\nu_1,\nu_2}(k+1,t_1,t_2)-\lambda(2k+1)q_{\nu_1,\nu_2}(k,t_1,t_2)\nonumber\\
	&\ \  +\lambda kq_{\nu_1,\nu_2}(k-1,t_1,t_2),\ k\ge0,
	\end{align}
with initial conditions $q_{\nu_1,\nu_2}(0,0,t_2)=q_{\nu_1,\nu_2}(0,0,t_2)=1$ for all $t_1\ge0$ and $t_2\ge0$. Here, the derivative involve in (\ref{pmfequ2}) is the Caputo fractional derivative as defined in (\ref{Caputoder}).

So, the pgf $\mathscr{G}_{\nu_1,\nu_2}(z,t_1,t_2)=\mathbb{E}z^{\mathscr{N}_{\nu_1,\nu_2}(t_1,t_2)}$, $|z|\leq1$ of FPRF solves 
\begin{equation}\label{fprfpgfeq}
	\frac{\partial^{\nu_1+\nu_2}}{\partial t_1^{\nu_1}\partial t_2^{\nu_2}}\mathscr{G}_{\nu_1,\nu_2}(z,t_1,t_2)=\lambda(z-1)^2\frac{\partial}{\partial z}\mathscr{G}_{\nu_1,\nu_2}(z,t_1,t_2)+\lambda(z-1)\mathscr{G}_{\nu_1,\nu_2}(z,t_1,t_2)
\end{equation}
with initial conditions $\mathscr{G}_{\nu_1,\nu_2}(z,0,t_2)=1$ and $\mathscr{G}_{\nu_1,\nu_2}(z,t_1,0)=1$.
On taking the derivative with respect to $z$ on both sides of (\ref{fprfpgfeq}) and substituting $z=1$, we obtain 
	\begin{equation}\label{meanequ}
		\frac{\partial^{\nu_1+\nu_2}}{\partial t_1^{\nu_1}\partial t_2^{\nu_2}}\mathbb{E}\mathscr{N}_{\nu_1,\nu_2}(t_1,t_2)=\lambda,
	\end{equation}
	with $\mathbb{E}\mathscr{N}_{\nu_1,\nu_2}(0,t_2)=\mathbb{E}\mathscr{N}_{\nu_1,\nu_2}(t_1,0)=0$. On solving (\ref{meanequ}), we get 
	\begin{equation}\label{fprfmean}
		\mathbb{E}\mathscr{N}_{\nu_1,\nu_2}(t_1,t_2)=\frac{\lambda t_1^{\nu_1}t_2^{\nu_2}}{\Gamma(\nu_1+1)\Gamma(\nu_2+1)}.
	\end{equation}
	
	 Let $\mu_2(t_1,t_2)\coloneqq \mathbb{E}\mathscr{N}_{\nu_1,\nu_2}(t_1,t_2)(\mathscr{N}_{\nu_1,\nu_2}(t_1,t_2)-1)$ be the second factorial moment of FPRF. On taking the derivative with respect to $z$ twice on both sides of (\ref{fprfpgfeq}) and substituting $z=1$, we get 
	\begin{equation}\label{2fact}
		\frac{\partial^{\nu_1+\nu_2}}{\partial t_1^{\nu_1}\partial t_2^{\nu_2}}\mu_2(t_1,t_2)=4\lambda\mathbb{E}\mathscr{N}_{\nu_1,\nu_2}(t_1,t_2),
	\end{equation}
	with $\mu_2(0,t_2)=\mu_2(t_1,0)=0$. On solving (\ref{2fact}), we get
	$$
		\mu_2(t_1,t_2)=\frac{(2\lambda t_1^{\nu_1}t_2^{\nu_2})^2}{\Gamma(2\nu_1+1)\Gamma(2\nu_2+1)}.
$$ Hence, the variance of FPRF is
\begin{equation}\label{varN}
	\mathbb{V}\mathrm{ar}\mathscr{N}_{\nu_1,\nu_2}(t_1,t_2)=\frac{\lambda t_1^{\nu_1}t_2^{\nu_2}}{\Gamma(\nu_1+1)\Gamma(\nu_2+1)}+\frac{(2\lambda t_1^{\nu_1}t_2^{\nu_2})^2}{\Gamma(2\nu_1+1)\Gamma(2\nu_2+1)}-\left(\frac{\lambda t_1^{\nu_1}t_2^{\nu_2}}{\Gamma(\nu_1+1)\Gamma(\nu_2+1)}\right)^2.
\end{equation}
\begin{remark}
	The mean and variance of FPRF given in (\ref{fprfmean}) and (\ref{varN}), respectively, coincide with the mean and variance of fractional Poisson field studied by Leonenko and Merzbach (2015). For $\nu_1=\nu_2=1$, the FPRF reduces to the PRF whose mean and variance are equal, that is,
	$
	\mathbb{E}\mathcal{N}(t_1,t_2)=\mathbb{V}\mathrm{ar}\mathcal{N}(t_1,t_2)=\lambda t_1t_2.
	$
\end{remark}

Next, we give a connection between the FPRF and PRF via a bivariate random process.
 Let $\{T_{2\nu_i}(t),\ t\ge0\}$, $0<\nu_i\leq1$, $i=1,2$ 
be two independent random processes whose corresponding densities are folded solutions of the following Cauchy problems:
\begin{equation}\label{couchyp}
		\frac{\partial^{2\nu_i}}{\partial t^{2\nu_i}}u_{2\nu_i}(x,t)=\frac{\partial^2}{\partial x^2}u_{2\nu_i}(x,t),\ t>0,\,x\in\mathbb{R},
		\end{equation}	
with initial condition $u_{2\nu_i}(x,0)=\delta(x),\ 0<\nu_i\leq1$ and also $(\partial/\partial t)u_{2\nu_i}(x,t)|_{t=0}=0,\ \frac{1}{2}<\nu_i\leq1$.
 Here, $\delta(x)$ is the Dirac delta function. On taking $c=1$ in Eq. (3.3) of Orsingher and Beghin (2004), the Laplace transform of $T_{2\nu_i}(t)$, $i=1,2$ is given by
\begin{equation}\label{fold}
	\int_{0}^{\infty}e^{-wt}\mathrm{Pr}\{T_{2\nu_i}(t)\in\mathrm{d}s\}\,\mathrm{d}t=w^{\nu_i-1}e^{-w^{\nu_i}s}\,\mathrm{d}s.
\end{equation}
\begin{theorem}\label{diffrelation}
Let $\{T_{2\nu_i}(t),\ t\ge0\}$, $i=1,2$ be two independent random processes whose Laplace transforms are given by (\ref{fold}), and let these be independent of the PRF $\{\mathcal{N}(t_1,t_2),\ (t_1,t_2)\in\mathbb{R}^2_+\}$. Then, the pgf of FPRF has following representation:
	\begin{equation}\label{pgf1}
		\mathscr{G}_{\nu_1,\nu_2}(z,t_1,t_2)=\int_{0}^{\infty}\int_{0}^{\infty}\mathcal{G}(z,s_1,s_2)\mathrm{Pr}\{T_{2\nu_1}(t_1)\in\mathrm{d}s_1\}\mathrm{Pr}\{T_{2\nu_2}(t_2)\in\mathrm{d}s_2\},
	\end{equation} 
	where $\mathcal{G}(z,s_1,s_2)$ is the pgf of PRF. Thus, 
	\begin{equation}\label{foldrelation}
		\mathscr{N}_{\nu_1,\nu_2}(t_1,t_2)\overset{d}{=} \mathcal{N}(T_{2\nu_1}(t_1),T_{2\nu_2}(t_2)),\ (t_1,t_2)\in\mathbb{R}^2_+,
	\end{equation}
	where $\overset{d}{=}$ denotes the equality in distribution.
\end{theorem}	
\begin{proof}
	On taking the Laplace transforms on the right hand side of (\ref{pgf1}) with respect to variables $t_1$ and $t_2$, successively, and using (\ref{fold}), we get
	\begin{align}\label{pgflap1}		
		\mathcal{L}(z,w_1,w_2)=w_1^{\nu_1-1}w_2^{\nu_2-1}\int_{0}^{\infty}\int_{0}^{\infty}\mathcal{G}(z,s_1,s_2)e^{-w_1^{\nu_1}s_1}e^{-w_2^{\nu_2}s_2}\,\mathrm{d}s_1\,\mathrm{d}s_2.
	\end{align}	
	On taking the Laplace transform with respect to $t_1$ on both sides of (\ref{pgf}), we get
	\begin{align*}
		w_1\frac{\partial}{\partial t_2}\int_{0}^{\infty}e^{-w_1t_1}\mathcal{G}(z,t_1,t_2)\,\mathrm{d}t_1&=\lambda(z-1)^2\frac{\partial}{\partial z}\int_{0}^{\infty}e^{-w_1t_1}\mathcal{G}(z,t_1,t_2)\,\mathrm{d}t_1\\
		&\ \ +\lambda(z-1)\int_{0}^{\infty}e^{-w_1t_1}\mathcal{G}(z,t_1,t_2)\,\mathrm{d}t_1,
	\end{align*}
	where we have used $\mathcal{G}(z,0,t_2)=1$. Its Laplace transform with respect to the variable $t_2$ gives
	\begin{align}\label{pgflap2}
		w_1w_2\int_{0}^{\infty}\int_{0}^{\infty}e^{-w_1t_1}e^{-w_2t_2}&\mathcal{G}(z,t_1,t_2)\,\mathrm{d}t_1\,\mathrm{d}t_2-1\nonumber\\
		&=\lambda(z-1)^2\frac{\partial}{\partial z}\int_{0}^{\infty}\int_{0}^{\infty}e^{-w_1t_1}e^{-w_2t_2}\mathcal{G}(z,t_1,t_2)\,\mathrm{d}t_1\,\mathrm{d}t_2\nonumber\\
		&\ \ +\lambda(z-1)\int_{0}^{\infty}\int_{0}^{\infty}e^{-w_1t_1}e^{-w_2t_2}\mathcal{G}(z,t_1,t_2)\,\mathrm{d}t_1\,\mathrm{d}t_2,
	\end{align}
	where we have used $\mathcal{G}(z,w_1,0)=1/w_1$. On using (\ref{pgflap1}) in (\ref{pgflap2}), we get
	\begin{equation}\label{cr1}
		w_1^{\nu_1}w_2^{\nu_2}\mathcal{L}(z,w_1,w_2)-w_1^{\nu_1-1}w_2^{\nu_2-1}=\lambda(z-1)^2\frac{\partial}{\partial z}\mathcal{L}(z,w_1,w_2)+\lambda(z-1)\mathcal{L}(z,w_1,w_2).
	\end{equation}
	Again, on taking the Laplace transform on both sides of (\ref{fprfpgfeq}) with respect to $t_1$ and $t_2$, successively, we get
	\begin{align}
		w_1^{\nu_1}w_2^{\nu_2}\int_{0}^{\infty}\int_{0}^{\infty}e^{-w_2t_2}e^{-w_1t_1}&\mathscr{G}_{\nu_1,\nu_2}(z,t_1,t_2)\,\mathrm{d}t_1\,\mathrm{d}t_2-w_1^{\nu_1-1}w_2^{\nu_2-1}\nonumber\\
		&=\lambda(z-1)^2\frac{\partial}{\partial z}\int_{0}^{\infty}\int_{0}^{\infty}e^{-w_2t_2}e^{-w_1t_1}\mathscr{G}_{\nu_1,\nu_2}(z,t_1,t_2)\,\mathrm{d}t_1\,\mathrm{d}t_2\nonumber\\
		&\ \ + \lambda(z-1)\int_{0}^{\infty}\int_{0}^{\infty}e^{-w_2t_2}e^{-w_1t_1}\mathscr{G}_{\nu_1,\nu_2}(z,t_1,t_2)\,\mathrm{d}t_1\,\mathrm{d}t_2.\label{pgflap3}
	\end{align}
 From (\ref{cr1}) and (\ref{pgflap3}), it can be observe that $\int_{0}^{\infty}\int_{0}^{\infty}e^{-w_2t_2}e^{-w_1t_1}\mathscr{G}_{\nu_1,\nu_2}(z,t_1,t_2)\,\mathrm{d}t_1\,\mathrm{d}t_2$ is a solution of (\ref{cr1}). The uniqueness of pgf completes the proof.
\end{proof}
{\begin{remark}
		For $i=1,2$, the solutions of Cauchy problems given in (\ref{couchyp}) are (see Podlubny (1999), p. 142)
		\begin{equation*}
			u_{2\nu_i}(x,t)=\frac{1}{2t^{\nu_i}}W_{-\nu_i,1-\nu_i}(-|x|t^{-\nu_i}),\ t>0,\,x\in\mathbb{R},
		\end{equation*} 
		where $W_{\alpha,\beta}(\cdot)$ is the Wright function as define in (\ref{wright}). So, the folded solution of (\ref{couchyp}) are given by (see Beghin and Orsingher (2009), Eq. (2.15))
		\begin{equation}\label{expfold}
			\tilde{u}_{2\nu_i}(x,t)=\begin{cases}
				2u_{2\nu_i}(x,t),\ x>0,\vspace{0.1cm}\\
				0,\ x<0.
			\end{cases}
		\end{equation}
		So, in view of (\ref{foldrelation}), the pmf of FPRF is given by
		\begin{equation*}
			q_{\nu_1,\nu_2}(k,t_1,t_2)=\frac{\lambda^k}{k!t_1^{\nu_1}t_2^{\nu_2}}\int_{0}^{\infty}\int_{0}^{\infty}e^{-\lambda s_1s_2}(s_1s_2)^kW_{-\nu_1,1-\nu_1}(-s_1t_1^{-\nu_1})W_{-\nu_2,1-\nu_2}(-s_2t_2^{-\nu_2})\,\mathrm{d}s_1\,\mathrm{d}s_2.
		\end{equation*}
\end{remark}}
\begin{remark}\label{rem1}
	From (\ref{foldrelation}), it can be observed that the FPRF is equal in distribution to a  time-changed PRF where the time is changed via a bivariate random process. Similar time-changed relationships have been established for some other fractional growth processes on positive real line, for example, the time fractional Poisson process (see Beghin and Orsingher (2009)), the fractional pure birth process (see Orsingher and Polito (2010)) and  the fractional linear birth-death process (see Orsingher and Polito (2011)), \textit{etc}.
	
	For $\nu_i\in(0,1)$, $i=1,2$, from (\ref{subdenlap1}) and (\ref{fold}), it follows that the distribution of $L_{\nu_i}(t)$ coincide with that of $T_{2\nu_i}(t)$ for $i=1,2$. So, the joint density function of $(L_{\nu_1}(t_1),L_{\nu_2}(t_2))$ coincide with that of $(T_{2\nu_1}(t_1),T_{2\nu_2}(t_2))$ for all $(t_1,t_2)\in\mathbb{R}^2_+$. Hence, $\mathcal{N}_{\nu_1,\nu_2}(t_1,t_2)\overset{d}{=}\mathscr{N}_{\nu_1,\nu_2}(t_1,t_2)$.
\end{remark}
\begin{remark}
	For some particular values of $\nu_i$, $i=1,2$, the distribution of $T_{2\nu_i}(t)$ can be explicitly obtained.  For example, if $\nu_i=\frac{1}{2^{r_i}}$,  $r_i\in\mathbb{N}$ then the density function of $T_{2\nu_i}(t)$ is same as that of folded $(r_i-1)$-iterated Brownian motion (see Orsingher and Beghin (2009), Theorem 2.2), that is, 
	\begin{align*}
		\mathrm{Pr}\{T_{({1}/{2})^{r_i-1}}(t)\in\mathrm{d}x\}&=\mathrm{Pr}\{|B^i_1(|B^i_2(\dots|B^i_{r_i}(t)|\dots)|)|\in\mathrm{d}x\}\\
		&=2^{r_i}\int_{0}^{\infty}\frac{e^{-x^2/4u_1}}{\sqrt{4\pi u_1}}\,\mathrm{d}u_1\int_{0}^{\infty}\frac{e^{-u_1^2/4u_2}}{\sqrt{4\pi u_2}}\,\mathrm{d}u_2\dots\int_{0}^{\infty}\frac{e^{-u_{r_i-1}^2/4t}}{\sqrt{4\pi t}}\,\mathrm{d}u_{r_i-1}\,\mathrm{d}x,
	\end{align*}
	where $\{B_j^i(t),\ t\ge0\}$, $j\in\{1,2,\dots,r_i\}$ are independent Brownian motions. 
	
	In particular, if $r_1=r_2=1$ then the distribution of $T_1(t)$ coincide with the distribution of reflecting Brownian motion. So, $\mathscr{N}_{\frac{1}{2},\frac{1}{2}}(t_1,t_2)$ is equal in distribution to $\mathcal{N}(|B^1(t_1)|,|B^2(t_2)|)$, where $\{B^1(t),\ t\ge0\}$ and $\{B^2(t),\ t\ge0\}$ are independent Brownian motions. Here, $B^i(t)$, $i=1,2$ are normal random variables with means zero and variances $2t$. Thus,
	\begin{equation*}
		\mathrm{Pr}\{\mathscr{N}_{\frac{1}{2},\frac{1}{2}}(t_1,t_2)=k\}=\frac{\lambda^k}{k!\pi\sqrt{t_1t_2}}\int_{0}^{\infty}\int_{0}^{\infty}{( w_1w_2)^k}\exp\left(-\frac{w_1^2}{4t_1}-\frac{w_2^2}{4t_2}-\lambda w_1w_2\right)\,\mathrm{d}w_1\,\mathrm{d}w_2.
	\end{equation*}	
	Moreover, for $\nu_i=1/3$, $i=1,2$, the density of $T_{2/3}(t)$ is given by (see Orsingher and Beghin (2009), Theorem 4.1)
	\begin{equation*}
		\mathrm{Pr}\{T_{2/3}(t)\in\mathrm{d}x\}/\mathrm{d}x=\frac{3}{2{(3t)}^{1/3}}\mathcal{A}i\left(\frac{|x|}{(3t)^{1/3}}\right),
	\end{equation*}
	where $\mathcal{A}i(\cdot)$ is the Airy function as defined in (\ref{airy}). Therefore,
	\begin{equation*}
		\mathrm{Pr}\{\mathscr{N}_{1/3,1/3}(t_1,t_2)=k\}=\frac{9\lambda^k}{4k!(9t_1t_2)^{1/3}}\int_{0}^{\infty}\int_{0}^{\infty}e^{-\lambda w_1w_2}(w_1w_2)^k\prod_{j=1}^{2}\mathcal{A}i\left(\frac{|w_j|}{(3t_j)^{1/3}}\right)\,\mathrm{d}w_1\,\mathrm{d}w_2.
	\end{equation*}
\end{remark}
\begin{proposition}
	Let $\nu_i\in(0,1)$, $i=1,2$. Then, for  a fixed $(\tau_1,\tau_2)\in\mathbb{R}^2_+$ such that $\tau_1<t_1$ and $\tau_2<t_2$, the covariance of $\mathscr{N}_{\nu_1,\nu_2}(\tau_1,\tau_2)$ and $\mathscr{N}_{\nu_1,\nu_2}(t_1,t_2)$ is given by
	\begin{equation}
		\mathbb{C}\mathrm{ov}(\mathscr{N}_{\nu_1,\nu_2}(\tau_1,\tau_2),\,\mathscr{N}_{\nu_1,\nu_2}(t_1,t_2))=\sum_{i=1}^{2}\frac{\lambda^i \tau_1^{\nu_1}\tau_2^{\nu_2}(- t_1^{\nu_1}t_2^{\nu_2})^{i-1}}{\Gamma^i(\nu_1+1)\Gamma^i(\nu_2+1)} +\lambda^2\prod_{i=1}^{2}\frac{\big(\tau_i^{2\nu_i}+t_i^{2\nu_i}I_{\tau_i/t_i}(\nu_i,\nu_i+1)\big)}{\Gamma(2\nu_i+1)},\label{cov1}
	\end{equation}
	where $I_x(\alpha,\beta)=B(x;\alpha,\beta)/B(\alpha,\beta)$, $0<x\leq1$. Here, $B(\alpha,\beta)$ and $B(x;\alpha,\beta)$, $\alpha>0$, $\beta>0$ are the beta and incomplete beta functions as defined in (\ref{beta}) and (\ref{ibeta}), respectively.
\end{proposition}
\begin{proof}
	In view of Remark \ref{rem1} and on using Eq. (26) of Leonenko  and Merzbach (2015), we get
	\begin{align*}
		\mathbb{C}\mathrm{ov}(&\mathscr{N}_{\nu_1,\nu_2}(\tau_1,\tau_2),\,\mathscr{N}_{\nu_1,\nu_2}(t_1,t_2))\\
		&=\frac{\lambda \tau_1^{\nu_1}\tau_2^{\nu_2}}{\Gamma(\nu_1+1)\Gamma(\nu_2+1)}-\frac{\lambda^2(\tau_1 t_1)^{\nu_1}(\tau_2 t_2)^{\nu_2}}{\Gamma^2(\nu_1+1)\Gamma^2(\nu_2+1)}+\frac{\lambda^2}{\nu_1\nu_2\Gamma^2(\nu_1)\Gamma^2(\nu_2)}\\
		&\ \ \cdot\int_{0}^{\tau_1}\int_{0}^{\tau_2}\left((\tau_1-x_1)^{\nu_1}+(t_1-x_1)^{\nu_1}\right)\left((\tau_2-x_2)^{\nu_2}+(t_2-x_2)^{\nu_2}\right)x_1^{\nu_1-1}x_2^{\nu_2-1}\,\mathrm{d}x_1\,\mathrm{d}x_2.
	\end{align*} 
	For $i=1,2$, we have
	\begin{align*}
		\int_{0}^{\tau_i}\big((\tau_i-x_i)^{\nu_i}+(t_i-x_i)^{\nu_i}\big)x_i^{\nu_i-1}\,\mathrm{d}x_i&=\tau_i^{2\nu_i}\int_{0}^{1}(1-u)^{\nu_i}u^{\nu_i-1}\,\mathrm{d}u+t_i^{2\nu_i}\int_{0}^{\tau_i/t_i}(1-u)^{\nu_i}u^{\nu_i-1}\,\mathrm{d}u\\
		&=\tau_i^{2\nu_i}B(\nu_i,\nu_i+1)+t_i^{2\nu_i}B(\tau_i/t_i;\nu_i,\nu_i+1)\\
		&={B(\nu_i,\nu_i+1)}({\tau_i^{2\nu_i}+t_i^{2\nu_i}I_{\tau_i/t_i}(\nu_i,\nu_i+1)}).
	\end{align*}
	So,
	\begin{align*}
		\mathbb{C}\mathrm{ov}(\mathscr{N}_{\nu_1,\nu_2}(\tau_1,\tau_2),\,\mathscr{N}_{\nu_1,\nu_2}(t_1,t_2))&=\frac{\lambda \tau_1^{\nu_1}\tau_2^{\nu_2}}{\Gamma(\nu_1+1)\Gamma(\nu_2+1)}-\frac{\lambda^2(\tau_1 t_1)^{\nu_1}(\tau_2 t_2)^{\nu_2}}{\Gamma^2(\nu_1+1)\Gamma^2(\nu_2+1)}\\
		&\ \ +\frac{\lambda^2}{\nu_1\nu_2\Gamma^2(\nu_1)\Gamma^2(\nu_2)} \prod_{i=1}^{2}B(\nu_i,\nu_i+1)({\tau_i^{2\nu_i}+t_i^{2\nu_i}I_{\tau_i/t_i}(\nu_i,\nu_i+1)}).
	\end{align*}
	On using (\ref{beta}), we get the required result.
\end{proof}
\begin{remark}
		On using (\ref{incbetalim}), we have the following limiting approximation of $I_{\tau_i/t_i}(\alpha,\beta)=B(\tau_i/t_i;\alpha,\beta)/B(\alpha,\beta)$:
		\begin{equation*}
		I_{\tau_i/t_i}(\nu_i,\nu_i+1)=\frac{1}{B(\nu_i,\nu_i+1)}\left(\frac{\tau_i^{\nu_i}}{\nu_it_i^{\nu_i}}-\frac{\nu_i\tau_i^{\nu_i+1}}{(\nu_i+1)t^{\nu_i+1}}\right)+O((\tau_i/t_i)^{\nu_i+2})\ \ \text{as}\ \ t_i\to\infty.
		\end{equation*}		
		Thus,
		\begin{equation*}
			t_i^{2\nu_i}I_{\tau_i/t_i}(\nu_i,\nu_i+1)\sim\frac{(\tau_it_i)^{\nu_i}}{\nu_iB(\nu_i,\nu_i+1)}\ \ \text{as}\ \ t_i\to\infty.
		\end{equation*}
		As $t_i\to\infty$, $i=1,2$, from (\ref{cov1}), we get the following asymptotic behaviour of the covariance of FPRF:
		{\footnotesize\begin{align*}
			\mathbb{C}\mathrm{ov}(\mathscr{N}_{\nu_1,\nu_2}(\tau_1,\tau_2),\,\mathscr{N}_{\nu_1,\nu_2}(t_1,t_2))&\sim\sum_{i=1}^{2}\frac{\lambda^i \tau_1^{\nu_1}\tau_2^{\nu_2}(-t_1^{\nu_1}t_2^{\nu_2})^{i-1}}{\Gamma^i(\nu_1+1)\Gamma^i(\nu_2+1)} +\lambda^2\prod_{i=1}^{2}\frac{1}{\Gamma(2\nu_i+1)}\left(\tau_i^{2\nu_i}+\frac{(\tau_it_i)^{\nu_i}}{\nu_iB(\nu_i,\nu_i+1)}\right)\\
			&=\sum_{i=1}^{2}\frac{\lambda^i \tau_1^{\nu_1}\tau_2^{\nu_2}(-t_1^{\nu_1}t_2^{\nu_2})^{i-1}}{\Gamma^i(\nu_1+1)\Gamma^i(\nu_2+1)} +\lambda^2\prod_{i=1}^{2}\left(\frac{\tau_i^{2\nu_i}}{\Gamma(2\nu_i+1)}+\frac{(\tau_it_i)^{\nu_i}}{\Gamma^2(\nu_i+1)}\right)\\
			&=\sum_{i=1}^{2}\frac{\lambda^i(\tau_1^{\nu_1}\tau_2^{\nu_2})^i}{\Gamma(i\nu_1+1)\Gamma(i\nu_2+1)}+\frac{\lambda^2\tau_1^{2\nu_1}(\tau_2t_2)^{\nu_2}}{\Gamma(2\nu_1+1)\Gamma^2(\nu_2+1)}+\frac{\lambda^2\tau_2^{2\nu_2}(\tau_1t_1)^{\nu_1}}{\Gamma(2\nu_2+1)\Gamma^2(\nu_1+1)}.
		\end{align*}}
\end{remark}
\subsection{Integrals of  Poisson random fields}
Orsingher and Polito (2013) studied the fractional integrals for the homogeneous Poisson process and its time changed variants. Here, we study these integrals for the PRF and FPRF over rectangle $[0,t_1]\times[0,t_2]$.

For $\alpha_1>0,\,\alpha_2>0$, we consider the fractional integral of FPRF as follows:
\begin{equation}\label{frintfprf}
	\mathscr{N}^{\alpha_1,\alpha_2}_{\nu_1,\nu_2}(t_1,t_2)=\frac{1}{\Gamma(\alpha_1)\Gamma(\alpha_2)}\int_{0}^{t_2}\int_{0}^{t_1}(t_2-\tau_2)^{\alpha_2-1}(t_1-\tau_1)^{\alpha_1-1}\mathscr{N}_{\nu_1,\nu_2}(\tau_1,\tau_2)\,\mathrm{d}\tau_2\,\mathrm{d}\tau_1.
\end{equation}
Its mean is given by
\begin{align*}
	\mathbb{E}\mathscr{N}^{\alpha_1,\alpha_2}_{\nu_1,\nu_2}(t_1,t_2)
	&=\frac{1}{\Gamma(\alpha_1)\Gamma(\alpha_2)}\int_{0}^{t_2}\int_{0}^{t_1}(t_2-\tau_2)^{\alpha_2-1}(t_1-\tau_1)^{\alpha_1-1}\mathbb{E}\mathscr{N}_{\nu_1,\nu_2}(\tau_1,\tau_2)\,\mathrm{d}\tau_2\,\mathrm{d}\tau_1\nonumber\\
	&=\frac{1}{\Gamma(\alpha_1)\Gamma(\alpha_2)}\frac{\lambda }{\Gamma(\nu_1+1)\Gamma(\nu_2+1)}\int_{0}^{t_2}\int_{0}^{t_1}(t_2-\tau_2)^{\alpha_2-1}(t_1-\tau_1)^{\alpha_1-1}\tau_1^{\nu_1}\tau_2^{\nu_2}\,\mathrm{d}\tau_2\,\mathrm{d}\tau_1\nonumber\\
	&=\frac{\lambda t_1^{\alpha_1+\nu_1}t_2^{\alpha_2+\nu_2}}{\Gamma(\alpha_1+\nu_1+1)\Gamma(\alpha_2+\nu_2+1)},
\end{align*}
where the penultimate step follows from (\ref{fprfmean}).

On taking $\alpha_1=\alpha_2=1$, (\ref{frintfprf}) reduces to the integral of FPRF whose mean is $\mathbb{E}\mathscr{N}_{\nu_1,\nu_2}^{1,1}(t_1,t_2)=\lambda t_1^{\nu_1+1}t_2^{\nu_2+1}/\Gamma(\nu_1+2)\Gamma(\nu_2+2)$. Also, for $\nu_1=\nu_2=1$, it reduces to the fractional integral of PRF, that is,
\begin{equation}\label{prfint}
	\mathcal{N}^{\alpha_1,\alpha_2}(t_1,t_2)=
	\frac{1}{\Gamma(\alpha_1)\Gamma(\alpha_2)}\int_{0}^{t_2}\int_{0}^{t_1}(t_2-\tau_2)^{\alpha_2-1}(t_1-\tau_1)^{\alpha_1-1}\mathcal{N}(\tau_1,\tau_2)\,\mathrm{d}\tau_2\,\mathrm{d}\tau_1,
\end{equation}
whose mean is 
\begin{equation}\label{meanint}
	\mathbb{E}\mathcal{N}^{\alpha_1,\alpha_2}(t_1,t_2)=\frac{\lambda t_1^{\alpha_1+1}t_2^{\alpha_2+1}}{\Gamma(\alpha_1+2)\Gamma(\alpha_2+2)}.
\end{equation}
Further, for $\alpha_1=\alpha_2=1$, (\ref{prfint}) reduces to the integral of PRF whose mean is $\mathbb{E}\mathcal{N}^{1,1}(t_1,t_2)=\lambda t_1^2t_2^2/4$.
\begin{theorem}
	The variance of fractional integral in (\ref{prfint}) is given by
	\begin{equation}\label{varinprf}
		\mathbb{V}\mathrm{ar}\mathcal{N}^{\alpha_1,\alpha_2}(t_1,t_2)=\lambda\prod_{i=1}^{2}\frac{ t_i^{2\alpha_i+1}}{(2\alpha_i+1)\Gamma^2(\alpha_i+1)}.
	\end{equation}
\end{theorem}
\begin{proof}
	For $(t_1,t_2),\,(\tau_1,\tau_2)\in\mathbb{R}^2_+$, we have the following result (see Leonenko and Merzbach (2015)):
	\begin{align*}
		\mathbb{E}\mathcal{N}(t_1,t_2)\mathcal{N}(\tau_1,\tau_2)&=\mathbb{V}\mathrm{ar}\mathcal{N}(1,1)\min\{t_1,\tau_1\}\min\{t_2,\tau_2\}+{(\mathbb{E}\mathcal{N}(1,1))^2}t_1t_2\tau_1\tau_2\\
		&=\lambda\min\{t_1,\tau_1\}\min\{t_2,\tau_2\}+\lambda^2t_1t_2\tau_1\tau_2.
	\end{align*}	
	The second moment of (\ref{prfint}) is given by
	\begin{align*}
		\mathbb{E}(\mathcal{N}^{\alpha_1,\alpha_2}(t_1,t_2))^2&=\frac{1}{(\Gamma(\alpha_1)\Gamma(\alpha_2))^2}\int_{0}^{t_2}\int_{0}^{t_1}\int_{0}^{t_2}\int_{0}^{t_1}(t_2-\tau_2)^{\alpha_2-1}(t_1-\tau_1)^{\alpha_1-1}\\
		&\ \ \cdot(t_2-\tau'_2)^{\alpha_2-1}(t_1-\tau'_1)^{\alpha_1-1}\mathbb{E}\mathcal{N}(\tau_1,\tau_2)\mathcal{N}(\tau'_1,\tau'_2)\,\mathrm{d}\tau_2\,\mathrm{d}\tau_1\mathrm{d}\tau'_2\,\mathrm{d}\tau'_1\\
		&=\frac{1}{(\Gamma(\alpha_1)\Gamma(\alpha_2))^2}\int_{0}^{t_2}\int_{0}^{t_1}\int_{0}^{t_2}\int_{0}^{t_1}(t_2-\tau_2)^{\alpha_2-1}(t_1-\tau_1)^{\alpha_1-1}(t_2-\tau'_2)^{\alpha_2-1}\\
		&\ \ \cdot(t_1-\tau'_1)^{\alpha_1-1}(\lambda\min\{\tau_1,\tau'_1\}\min\{\tau_2,\tau'_2\}+\lambda^2\tau_1\tau_2\tau'_1\tau'_2)\,\mathrm{d}\tau_2\,\mathrm{d}\tau_1\mathrm{d}\tau'_2\,\mathrm{d}\tau'_1\\
		&=\frac{\lambda^2t_1^{2\alpha_1+2}t_2^{2\alpha_2+2}}{\Gamma^2(\alpha_1+2)\Gamma^2(\alpha_2+2)}+\frac{\lambda}{(\Gamma(\alpha_1)\Gamma(\alpha_2))^2}\int_{0}^{t_2}\int_{0}^{t_1}\int_{0}^{t_2}\int_{0}^{t_1}(t_2-\tau_2)^{\alpha_2-1}\\
		&\ \ \cdot(t_1-\tau_1)^{\alpha_1-1}(t_2-\tau'_2)^{\alpha_2-1}(t_1-\tau'_1)^{\alpha_1-1}\min\{\tau_1,\tau'_1\}\min\{\tau_2,\tau'_2\}\,\mathrm{d}\tau_2\,\mathrm{d}\tau_1\mathrm{d}\tau'_2\,\mathrm{d}\tau'_1\\
		&=\lambda^2\prod_{i=1}^{2}\frac{ t_i^{2\alpha_i+2}}{\Gamma^2(\alpha_i+2)}+\lambda\prod_{i=1}^{2}\frac{1}{\Gamma^2(\alpha_i)}\int_{0}^{t_i}(t_i-\tau_i)^{\alpha_i-1}\\
		&\ \ \cdot\left(\int_{0}^{\tau_i}(t_i-\tau'_i)^{\alpha_i-1}\tau'_i\,\mathrm{d}\tau'_i+\int_{\tau_i}^{t_i}(t_i-\tau'_i)^{\alpha_i-1}\tau_i\,\mathrm{d}\tau'_i\right)\,\mathrm{d}\tau_i\\
		&=\lambda^2\prod_{i=1}^{2}\frac{ t_i^{2\alpha_i+2}}{\Gamma^2(\alpha_i+2)}+\lambda\prod_{i=1}^{2}\frac{1}{\Gamma(\alpha_i)\Gamma(\alpha_i+2)}\\
		&\ \ \cdot\int_{0}^{t_i}(t_i-\tau_i)^{\alpha_i-1}(t_i^{\alpha_i+1}-(t_i-\tau_i)^{\alpha_i+1})\,\mathrm{d}\tau_i\\
		&=\lambda^2\prod_{i=1}^{2}\frac{ t_i^{2\alpha_i+2}}{\Gamma^2(\alpha_i+2)}+\lambda\prod_{i=1}^{2}\frac{t_i^{2\alpha_i+1}}{(2\alpha_i+1)\Gamma^2(\alpha_i+1)}.
	\end{align*}
	On using (\ref{meanint}), we get the required result. 
\end{proof}
\begin{remark}
	On taking $\alpha_1=\alpha_2=1$ in (\ref{varinprf}), we get $\mathbb{V}\mathrm{ar}\mathcal{N}^{1,1}(t_1,t_2)=\lambda t_1^3t_2^3/9$.
\end{remark}
The following result will be used (see Baddeley (2007)):
\begin{lemma}
	Let $\{\mathcal{N}(t_1,t_2),\ (t_1,t_2)\in\mathbb{R}^2_+\}$ be the PRF with parameter $\lambda>0$. Then, conditional on the event $\{\mathcal{N}(t_1,t_2)=l\}$, $\mathcal{N}(\tau_1,\tau_2)$ has binomial distribution for $(\tau_1,\tau_2)\prec(t_1,t_2)$. That is,
	\begin{equation}\label{conddist}
		\mathrm{Pr}\{\mathcal{N}(\tau_1,\tau_2)=k|\mathcal{N}(t_1,t_2)=l\}=\binom{l}{k}\left(\frac{\tau_1\tau_2}{t_1t_2}\right)^k\left(1-\frac{\tau_1\tau_2}{t_1t_2}\right)^{l-k},\ k=0,1,2,\dots,l.
	\end{equation}
\end{lemma}

\begin{remark}
	On using (\ref{conddist}), we get $\mathbb{E}\{\mathcal{N}(\tau_1,\tau_2)|\mathcal{N}(t_1,t_2)=l\}=l\tau_1\tau_2/t_1t_2$. Thus, the conditional mean of (\ref{prfint}) is given by
	\begin{equation*}
		\mathbb{E}\{\mathcal{N}^{\alpha_1,\alpha_2}(t_1,t_2)|\mathcal{N}(t_1,t_2)=l\}=l\prod_{i=1}^{2}\frac{t_i^{-1}}{\Gamma(\alpha_i)}\int_{0}^{t_i}(t_i-\tau_i)^{\alpha_i-1}{\tau_i}\,\mathrm{d}\tau_i=\frac{lt_1^{\alpha_1}t_2^{\alpha_2}}{\Gamma(\alpha_1+2)\Gamma(\alpha_2+2)}.
	\end{equation*}
\end{remark} 
\section{Statistical properties of FPRF}\label{sec4}
Here, we study some distributional properties of the FPRF. First, we obtain an explicit expression of the distribution of FPRF using the Adomian decomposition method.

Note that the Riemann-Liouville integral operator $I_t^\nu$, $\nu>0$ as defined in (\ref{RL}) is linear. So, the Adomian polynomials associated with it are given by $I_t^\nu(u_k(t))=A_k(u_0(t),u_1(t),\dots,u_k(t))$.	
On applying the Riemann-Liouville integral operator $I_{t_1}^{\nu_1}$, $0<\nu_1\leq1$ on both sides of (\ref{pmfequ2}), we get
		\begin{align}\label{pmfequ3}
		\frac{\partial^{\nu_2}}{\partial t_2^{\nu_2}}q_{\nu_1,\nu_2}(k,t_1,t_2)&=\lambda(k+1)I_{t_1}^{\nu_1}q_{\nu_1,\nu_2}(k+1,t_1,t_2)-\lambda(2k+1)I_{t_1}^{\nu_1}q_{\nu_1,\nu_2}(k,t_1,t_2)\nonumber\\
		&\ \  +\lambda kI_{t_1}^{\nu_1}q_{\nu_1,\nu_2}(k-1,t_1,t_2),\ k\ge0,
	\end{align}
where we have used $(\partial^{\nu_2}/\partial t_2^{\nu_2})q_{\nu_1,\nu_2}(k,0,t_2)=0$ for all $k\ge0$. Further, on applying $I_{t_2}^{\nu_2}$, $0<\nu_2\leq1$ on both sides of (\ref{pmfequ3}) and substituting $q_{\nu_1,\nu_2}(k,t_1,t_2)=\sum_{n=0}^{\infty}q_{\nu_1,\nu_2}^n(k,t_1,t_2)$, $k\ge0$, we get
	
\begin{align*}\label{pmfequ4}
	\sum_{n=0}^{\infty}q_{\nu_1,\nu_2}^n(k,t_1,t_2)&=q_{\nu_1,\nu_2}(k,t_1,0)+\sum_{n=0}^{\infty}I_{t_2}^{\nu_2}I_{t_1}^{\nu_1}\big(\lambda(k+1)q_{\nu_1,\nu_2}^n(k+1,t_1,t_2)\nonumber\\
	&\ \ -\lambda(2k+1)q_{\nu_1,\nu_2}^n(k,t_1,t_2) +\lambda kq_{\nu_1,\nu_2}^n(k-1,t_1,t_2)\big),\ k\ge0.
\end{align*}	
On applying the Adomian decomposition method, we obtain
	\begin{equation}\label{adm1}
		q^0_{\nu_1,\nu_2}(k,t_1,t_2)=	q_{\nu_1,\nu_2}(k,t_1,0)=\begin{cases}
			1,\ k=0,\\
			0,\ k\ne0 
		\end{cases}
	\end{equation}
and
\begin{equation}\label{adm2}
	q^n_{\nu_1,\nu_2}(k,t_1,t_2)=\begin{cases}
		I_{t_2}^{\nu_2}I_{t_1}^{\nu_1}\big(\lambda q_{\nu_1,\nu_2}^{n-1}(1,t_1,t_2)-\lambda q_{\nu_1,\nu_2}^{n-1}(0,t_1,t_2)\big),\ k=0,\,n\ge1,\vspace{0.15cm}\\
		I_{t_2}^{\nu_2}I_{t_1}^{\nu_1}\big(\lambda(k+1)q_{\nu_1,\nu_2}^{n-1}(k+1,t_1,t_2) -\lambda(2k+1)q_{\nu_1,\nu_2}^{n-1}(k,t_1,t_2)\\
		\hspace{4.5cm}+\lambda kq_{\nu_1,\nu_2}^{n-1}(k-1,t_1,t_2)\big),\ k\ge1,\,n\ge1.
	\end{cases}
\end{equation}

Next result provides a sufficient condition for the series components $q_{\nu_1,\nu_2}^n(k,t_1,t_2)$ to vanish.
\begin{proposition}\label{prop1}    
	For $t\ge0$, the series components as defined by (\ref{adm2}) vanish for all $k\ge n+1$, that is, $q_{\nu_1,\nu_2}^n(k,t_1,t_2)=0$ provided $k-n\ge1$.
\end{proposition}
\begin{proof}
	It is sufficient to show that
	\begin{equation}\label{prop}
		q_{\nu_1,\nu_2}^n(k+n,t_1,t_2)=0,
	\end{equation}
	for all $k\ge1$ and $n\ge0$. From (\ref{adm1}), we note that (\ref{prop}) holds for $n=0$ and $k\geq1$. 	
	
	For $k=n=1$, from (\ref{adm2}), we get
	\begin{equation*}
		\begin{split}            p_{\nu_1,\nu_2}^1(2,t_1,t_2)&=I_{t_2}^{\nu_2}I_{t_1}^{\nu_1}(3\lambda q_{\nu_1,\nu_2}^0(3,t_1,t_2)-5\lambda q_{\nu_1,\nu_2}^0(2,t_1,t_2)+2\lambda q_{\nu_1,\nu_2}^0(1,t_1,t_2))=0,      
		\end{split}
	\end{equation*}
	where we have used (\ref{adm1}). So, \((\ref{prop})\) holds for \(k=n=1\).
	
	 If $n=1$ and $k\ge1$ then 
	\begin{align*}
		p_{\nu_1,\nu_2}^1(k+1,t_1,t_2)&=I_{t_2}^{\nu_2}I_{t_1}^{\nu_1}\big(\lambda(k+2)q_{\nu_1,\nu_2}^{0}(k+2,t_1,t_2) -\lambda(2k+3)q_{\nu_1,\nu_2}^{0}(k+1,t_1,t_2)\\
		&\ \ +\lambda (k+1)q_{\nu_1,\nu_2}^{0}(k,t_1,t_2)\big)=I_{t_2}^{\nu_2}I_{t_1}^{\nu_1}(0)=0.
	\end{align*}
	Thus, (\ref{prop}) holds for  $n=1$ and $k\ge1$. 
	
	Now, suppose  the result holds for some $n=m$, $m\ge2$ and for all $k\ge1$, that is, 
	$
		q_{\nu_1,\nu_2}^m(m+k,t)=0.
$
	Then, from (\ref{adm2}), we have
	\begin{align*}
			q_{\nu_1,\nu_2}^{m+1}(m+k+1,t)&=I_{t_2}^{\nu_2}I_{t_2}^{\nu_2}(-\lambda(m+k+2)q_{\nu_1,\nu_2}^m(m+k+2,t_1,t_2)\\
			&\ \ -\lambda(2(m+k)+3)q_{\nu_1,\nu_2}^m(m+k+1,t_1,t_2)\\&\ \ +\lambda(m+k+1)q_{\nu_1,\nu_2}^m(m+k,t_1,t_2))=I_{t_2}^{\nu_2}I_{t_1}^{\nu_1}(0)=0,         
	\end{align*}
	where we have used the induction hypothesis to get the penultimate step. Thus, the proof completes using the method of induction.
\end{proof}
Now, we explicitly obtain some series components of the pmf of FPRF. From (\ref{adm1}), we have $p_{\nu_1,\nu_2}^0(0,t_1,t_2)=1$. On taking $k=0$ in (\ref{adm2}) and using Proposition \ref{prop1}, we get
\begin{equation}\label{0.0}
	q_{\nu_1,\nu_2}^1(0,t_1,t_2)=I_{t_2}^{\nu_2}I_{t_1}^{\nu_1}(-\lambda t_1^0t_2^0)=\frac{-\lambda t_1^{\nu_1}t_2^{\nu_2}}{\Gamma(\nu_1+1)\Gamma(\nu_2+1)}.
\end{equation}
On taking $k=1$ in (\ref{adm2}), we get
\begin{align}
	q_{\nu_1,\nu_2}^1(1,t_1,t_2)&=I_{t_2}^{\nu_2}I_{t_1}^{\nu_1}(2\lambda q_{\nu_1,\nu_2}^0(2,t_1,t_2)-3\lambda q_{\nu_1,\nu_2}^0(1,t_1,t_2)+\lambda q_{\nu_1,\nu_2}^0(0,t_1,t_2))\nonumber\\
	&= I_{t_2}^{\nu_2}I_{t_1}^{\nu_1}(\lambda t_1^0t_2^0)=\frac{\lambda t_1^{\nu_1}t_2^{\nu_2}}{\Gamma(\nu_1+1)\Gamma(\nu_2+1)},\label{1.0}
\end{align}
where the penultimate step follows from (\ref{adm1}). So,
\begin{align*}
	q_{\nu_1,\nu_2}^2(0,t_1,t_2)&=I_{t_2}^{\nu_2}I_{t_1}^{\nu_1}(\lambda q_{\nu_1,\nu_2}^{1}(1,t_1,t_2)-\lambda q_{\nu_1,\nu_2}^{1}(0,t_1,t_2))\\
	&=\frac{2\lambda^2 I_{t_2}^{\nu_2}I_{t_1}^{\nu_1}(t_1^{\nu_1}t_2^{\nu_2})}{\Gamma(\nu_1+1)\Gamma(\nu_2+1)}=\frac{2(\lambda t_1^{\nu_1}t_2^{\nu_2})^2}{\Gamma(2\nu_1+1)\Gamma(2\nu_2+1)}.
\end{align*}
From Proposition \ref{prop1}, we have $q_{\nu_1,\nu_2}^1(2,t_1,t_2)=0$. Thus,
\begin{align*}
	q_{\nu_1,\nu_2}^2(1,t_1,t_2)&=	I_{t_2}^{\nu_2}I_{t_1}^{\nu_1}(2\lambda q_{\nu_1,\nu_2}^{1}(2,t_1,t_2) -3\lambda q_{\nu_1,\nu_2}^{1}(1,t_1,t_2) +\lambda q_{\nu_1,\nu_2}^{1}(0,t_1,t_2))\\
	&=-\frac{4\lambda^2I_{t_2}^{\nu_2}I_{t_1}^{\nu_1}(t_1^{\nu_1}t_2^{\nu_2})}{\Gamma(\nu_1+1)\Gamma(\nu_2+1)}=-\frac{4(\lambda t_1^{\nu_1}t_2^{\nu_2})^2}{\Gamma(2\nu_1+1)\Gamma(2\nu_2+1)},
\end{align*}
where we have used (\ref{0.0}) and (\ref{1.0}) to get the penultimate step.

Proceeding in the similar way, we obtain
\begin{align*}
	q_{\nu_1,\nu_2}^3(0,t_1,t_2)&=I_{t_2}^{\nu_2}I_{t_1}^{\nu_1}(\lambda q_{\nu_1,\nu_2}^2(1,t_1,t_2)-\lambda q_{\nu_1,\nu_2}^2(0,t_1,t_2))=\frac{6(-\lambda t_1^{\nu_1}t_2^{\nu_2})^3}{\Gamma(3\nu_1+1)\Gamma(3\nu_2+1)},\\
	q_{\nu_1,\nu_2}^2(2,t_1,t_2)&=I_{t_2}^{\nu_2}I_{t_1}^{\nu_1}(3\lambda q_{\nu_1,\nu_2}^{1}(3,t_1,t_2) -5\lambda q_{\nu_1,\nu_2}^{1}(2,t_1,t_2)+2\lambda q_{\nu_1,\nu_2}^{1}(1,t_1,t_2))\\
	&=\frac{2(\lambda t_1^{\nu_1}t_2^{\nu_2})^2}{\Gamma(2\nu_1+1)\Gamma(2\nu_2+1)},\\
	q_{\nu_1,\nu_2}^3(1,t_1,t_2)&=I_{t_2}^{\nu_2}I_{t_1}^{\nu_1}(2\lambda q_{\nu_1,\nu_2}^{2}(2,t_1,t_2) -3\lambda q_{\nu_1,\nu_2}^{2}(1,t_1,t_2) +\lambda q_{\nu_1,\nu_2}^{2}(0,t_1,t_2))\\
	&=\frac{18(\lambda t_1^{\nu_1}t_2^{\nu_2})^3}{\Gamma(3\nu_1+1)\Gamma(3\nu_2+1)},\\
	q_{\nu_1,\nu_2}^4(0,t_1,t_2)&=I_{t_2}^{\nu_2}I_{t_1}^{\nu_1}(\lambda q_{\nu_1,\nu_2}^3(1,t_1,t_2)-\lambda q_{\nu_1,\nu_2}^3(0,t_1,t_2))=\frac{24(\lambda t_1^{\nu_1}t_2^{\nu_2})^4}{\Gamma(4\nu_1+1)\Gamma(4\nu_2+1)}.
\end{align*}

The observed pattern of the appearance of series components are illustrated in Figure \ref{fig}. Note that to get $q_{\nu_1,\nu_2}^n(0,t_1,t_2)$, $n\ge1$, we require the series components $q_{\nu_1,\nu_2}^{n-1}(0,t_1,t_2)$ and $q_{\nu_1,\nu_2}^{n-1}(1,t_1,t_2)$. For $k\ge1$, to obtain $q_{\nu_1,\nu_2}^n(k,t_1,t_2)$, $n\ge1$, we need the series components $q_{\nu_1,\nu_2}^{n-1}(k-1,t_1,t_2)$, $q_{\nu_1,\nu_2}^{n-1}(k,t_1,t_2)$ and $q_{\nu_1,\nu_2}^{n-1}(k+1,t_1,t_2)$.

\begin{figure}[ht!]
	\includegraphics[width=14cm]{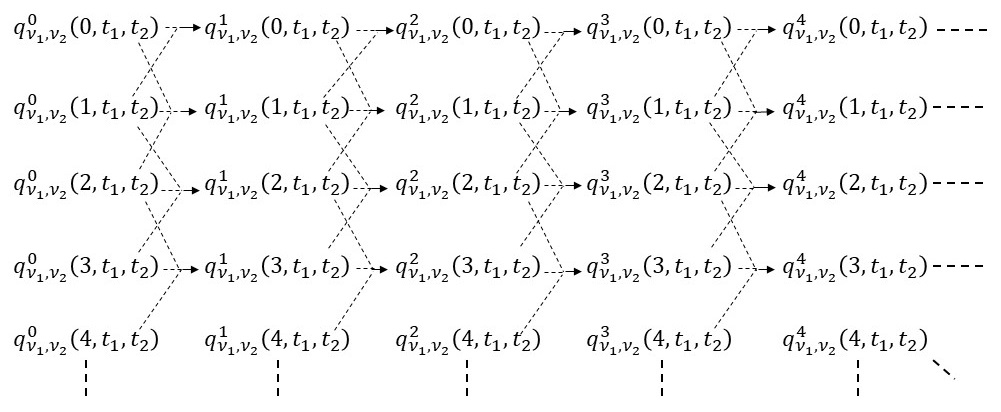}
	\caption{Pattern in the series components}\label{fig}
\end{figure}

In the next result, we obtain a closed form of the series components of the pmf of FPRF.
\begin{theorem}
	The series components of the pmf of FPRF are given by
	\begin{equation}\label{pmfcompts}
		q_{\nu_1,\nu_2}^n(k,t_1,t_2)=\begin{cases}
			\displaystyle\frac{n!(-\lambda t_1^{\nu_1}t_2^{\nu_2})^n}{\Gamma(n\nu_1+1)\Gamma(n\nu_2+1)},\ k=0,\,n\ge0,\vspace{0.15cm}\\
			\displaystyle\frac{(-1)^{n-k}n_{(n-k)}n_{(k)}(\lambda t_1^{\nu_1}t_2^{\nu_2})^{n}}{\Gamma(n\nu_1+1)\Gamma(n\nu_2+1)},\ n\ge k\ge1,\vspace{0.15cm}\\
			0,\ k-n\ge1,
		\end{cases}
	\end{equation}
	where $n_{(k)}=n(n-1)(n-2)\dots(n-k+1)$.
\end{theorem}
\begin{proof}
	From Proposition \ref{prop1}, it follows that $q_{\nu_1,\nu_2}^n(k,t_1,t_2)=0$ whenever $k-n\ge1$. From (\ref{1.0}), the result holds true for $k=n=1$. 
	
	Let us assume that (\ref{pmfcompts}) holds for any $k\geq0$ and for some $n=m\ge2$, that is,
	\begin{equation}\label{indhy1}
	q_{\nu_1,\nu_2}^m(k,t_1,t_2)=\begin{cases}
		\displaystyle\frac{m!(-\lambda t_1^{\nu_1}t_2^{\nu_2})^m}{\Gamma(m\nu_1+1)\Gamma(m\nu_2+1)},\ k=0,\,m\ge2,\vspace{0.15cm}\\
		\displaystyle\frac{(-1)^{m-k}m_{(m-k)}m_{(k)}(\lambda t_1^{\nu_1}t_2^{\nu_2})^{m}}{\Gamma(m\nu_1+1)\Gamma(m\nu_2+1)},\ m\ge k\ge1,\vspace{0.15cm}\\
		0,\ k-m\ge1.
	\end{cases}
	\end{equation}
	Then, we have the following three cases:
	
	\paragraph{\textit{Case I}} For $k=0$, from (\ref{adm2}), we get
	\begin{align*}
		q_{\nu_1,\nu_2}^{m+1}(0,t_1,t_2)&=I_{t_2}^{\nu_2}I_{t_1}^{\nu_1}(\lambda q_{\nu_1,\nu_2}^{m}(1,t_1,t_2)-\lambda q_{\nu_1,\nu_2}^{m}(0,t_1,t_2))\\
		&=\frac{(m+1)!(-\lambda)^{m+1}}{\Gamma(m\nu_1+1)\Gamma(m\nu_2+1)}I_{t_2}^{\nu_2}I_{t_1}^{\nu_1}(t_1^{m\nu_1}t_2^{m\nu_2})=\frac{(m+1)!(-\lambda t_1^{\nu_1}t_2^{\nu_2})^{m+1}}{\Gamma((m+1)\nu_1+1)\Gamma((m+1)\nu_2+1)},
	\end{align*}
	where we have used the induction hypothesis (\ref{indhy1}) to get the penultimate step. So, (\ref{pmfcompts}) holds for $k=0$ and $n=m+1$.
	\paragraph{\textit{Case II}} If $k=1$ then from (\ref{adm2}), we get
	\begin{align*}
		q_{\nu_1,\nu_2}^{m+1}(1,t_1,t_2)&=I_{t_2}^{\nu_2}I_{t_1}^{\nu_1}(2\lambda q_{\nu_1,\nu_2}^{m}(2,t_1,t_2) -3\lambda q_{\nu_1,\nu_2}^{m}(1,t_1,t_2)+\lambda q_{\nu_1,\nu_2}^{m}(0,t_1,t_2))\\
		&=\frac{(-1)^mm!(m+1)^2(\lambda t_1^{\nu_1}t_2^{\nu_2})^{m+1}}{\Gamma((m+1)\nu_1+1)\Gamma((m+1)\nu_2+1)}.
	\end{align*}
	Hence, (\ref{pmfcompts}) holds for $k=1$ and $n=m+1$ also.
	\paragraph{\textit{Case III}} Let $k\ge2$. If $k-m>1$ then the result follows from Proposition \ref{prop}.
	
	Suppose $k-m=1$. Then, $q_{\nu_1,\nu_2}^m(k,t_1,t_2)=q_{\nu_1,\nu_2}^m(k+1,t_1,t_2)=0$ and from (\ref{adm2}), we get
	\begin{align*}
		q_{\nu_1,\nu_2}^{m+1}(k,t_1,t_2)&=I_{t_1}^{\nu_1}I_{t_2}^{\nu_2}((m+1)\lambda q_{\nu_1,\nu_2}^m(m,t_1,t_2))\\
		&=\frac{(m+1)m_{(0)}m_{(m)}\lambda^{m+1}I_{t_2}^{\nu_2}I_{t_1}^{\nu_1}( t_1^{m\nu_1}t_2^{m\nu_2})}{\Gamma(m\nu_1+1)\Gamma(m\nu_2+1)}=\frac{(m+1)!(\lambda t_1^{\nu_1}t_2^{\nu_2})^{m+1}}{\Gamma((m+1)\nu_1+1)\Gamma((m+1)\nu_2+1)},
	\end{align*}
	where we have used $k=m+1$. 
	
	If $k-m=0$ then $q_{\nu_1,\nu_2}^m(k+1,t_1,t_2)=0$ and 
	\begin{align*}
		q_{\nu_1,\nu_2}^{m+1}(k,t_1,t_2)&=I_{t_2}^{\nu_2}I_{t_1}^{\nu_1}( -\lambda(2m+1)q_{\nu_1,\nu_2}^{m}(m,t_1,t_2)+\lambda mq_{\nu_1,\nu_2}^{m}(m-1,t_1,t_2))\\
		&=-\frac{(m+1)!(m+1)(\lambda t_1^{\nu_1}t_2^{\nu_2})^{m+1}}{\Gamma((m+1)\nu_1+1)\Gamma((m+1)\nu_2+1)}.
	\end{align*}
	
	For $k-m<0$, from (\ref{adm2}), we have
	\begin{align*}
		q_{\nu_1,\nu_2}^{m+1}(k,t_1,t_2)&=I_{t_2}^{\nu_2}I_{t_1}^{\nu_1}(\lambda(k+1)q_{\nu_1,\nu_2}^{m}(k+1,t_1,t_2)\\
		&\ \  -\lambda(2k+1)q_{\nu_1,\nu_2}^{m}(k,t_1,t_2)+\lambda kq_{\nu_1,\nu_2}^{m}(k-1,t_1,t_2))\\
		&=((k+1)m_{(m-k-1)}m_{(k+1)}+(2k+1)m_{(m-k)}m_{(k)}+km_{(m-k+1)}m_{(k-1)})\\
		&\ \ \cdot\frac{(-1)^{m+1-k}(\lambda t_1^{\nu_1}t_2^{\nu_2})^{m+1}}{\Gamma((m+1)\nu_1+1)\Gamma((m+1)\nu_2+1)}\\
		&=\left(\frac{(k+1)m!m!}{(k+1)!(m-k-1)!}+\frac{(2k+1)m!m!}{k!(m-k)!}+\frac{km!m!}{(k-1)!(m-k+1)!}\right)\\
		&\ \ \cdot\frac{(-1)^{m+1-k}(\lambda t_1^{\nu_1}t_2^{\nu_2})^{m+1}}{\Gamma((m+1)\nu_1+1)\Gamma((m+1)\nu_2+1)}\\
		&=\frac{(-1)^{m+1-k}m!m!(m+1)^2(\lambda t_1^{\nu_1}t_2^{\nu_2})^{m+1}}{k!(m-k+1)!\Gamma((m+1)\nu_1+1)\Gamma((m+1)\nu_2+1)}\\
		&=\frac{(-1)^{m+1-k}(m+1)_{(m+1-k)}(m+1)_{(k)}(\lambda t_1^{\nu_1}t_2^{\nu_2})^{m+1}}{\Gamma((m+1)\nu_1+1)\Gamma((m+1)\nu_2+1)}.
	\end{align*}
	Therefore, (\ref{pmfcompts}) holds for $k\ge2$ and $n=m+1$. The proof is complete using the method of induction.
	\end{proof}
	On summing (\ref{pmfcompts}) over $n=0,1,2,\dots$, we get the following result:
\begin{theorem}\label{rem2}
	The void probability of FPRF is given by
	\begin{equation}\label{v1}
		q_{\nu_1,\nu_2}(0,t_1,t_2)=\sum_{n=0}^{\infty}\frac{n!(-\lambda t_1^{\nu_1}t_2^{\nu_2})^n}{\Gamma(n\nu_1+1)\Gamma(n\nu_2+1)}.
	\end{equation}
	For $k\ge1$, its state probabilities are
	\begin{equation}\label{p1}
		q_{\nu_1,\nu_2}(k,t_1,t_2)=\sum_{n=k}^{\infty}\frac{(-1)^{n-k}n_{(n-k)}n_{(k)}(\lambda t_1^{\nu_1}t_2^{\nu_2})^{n}}{\Gamma(n\nu_1+1)\Gamma(n\nu_2+1)}.
	\end{equation}
\end{theorem}
\begin{remark}
	The state probabilities of FPRF satisfy the regularity condition, that is,
	\begin{align*}
		\sum_{k=0}^{\infty}q_{\nu_1,\nu_2}(k,t_1,t_2)&=\sum_{n=0}^{\infty}\frac{n!(-\lambda t_1^{\nu_1}t_2^{\nu_2})^n}{\Gamma(n\nu_1+1)\Gamma(n\nu_2+1)}+\sum_{k=1}^{\infty}\sum_{n=k}^{\infty}\frac{(-1)^{n-k}n_{(n-k)}n_{(k)}(\lambda t_1^{\nu_1}t_2^{\nu_2})^{n}}{\Gamma(n\nu_1+1)\Gamma(n\nu_2+1)}\\
		&=\sum_{n=0}^{\infty}\frac{n!(-\lambda t_1^{\nu_1}t_2^{\nu_2})^n}{\Gamma(n\nu_1+1)\Gamma(n\nu_2+1)}+\sum_{n=1}^{\infty}\sum_{k=1}^{n}\frac{(-1)^{n-k}n_{(n-k)}n_{(k)}(\lambda t_1^{\nu_1}t_2^{\nu_2})^{n}}{\Gamma(n\nu_1+1)\Gamma(n\nu_2+1)}\\
		&=1+\sum_{n=1}^{\infty}\frac{n!(-\lambda t_1^{\nu_1}t_2^{\nu_2})^n}{\Gamma(n\nu_1+1)\Gamma(n\nu_2+1)}\bigg(1+\sum_{k=1}^{n}\binom{n}{k}(-1)^{k}\bigg)=1.
	\end{align*}
\end{remark}
\begin{remark}
	The pmf of FPRF has the following equivalent representation:
	\begin{align}
		q_{\nu_1,\nu_2}(k,t_1,t_2)&=\frac{(\lambda t_1^{\nu_1}t_2^{\nu_2})^k}{k!}\sum_{n=0}^{\infty}\frac{\Gamma(n+k+1)\Gamma(n+k+1)(-\lambda t_1^{\nu_1}t_2^{\nu_2})^{n}}{\Gamma(n\nu_1+k\nu_1+1)\Gamma(n\nu_2+k\nu_2+1)n!}\nonumber\\
		&=\frac{(\lambda t_1^{\nu_1}t_2^{\nu_2})^k}{k!}{}_2\Psi_2\left[\begin{matrix}
			(k+1,1),\ (k+1,1)\\\\
			(k\nu_1+1,\nu_1),\ (k\nu_2+1,\nu_2)
		\end{matrix}\Bigg| -\lambda t_1^{\nu_1}t_2^{\nu_2}\right],\ k\ge0,\label{c*}
	\end{align}
	where ${}_2\Psi_2(\cdot)$ is the generalized Wright function defined in (\ref{genwrit}).
	For $\nu_1=\nu_2=1$, (\ref{v1}) and (\ref{p1}) reduces to
	$q(0,t_1,t_2)=e^{-\lambda t_1t_2}$ and $q(k,t_1,t_2)={(\lambda t_1t_2)^ke^{-\lambda t_1t_2}}/{k!}$, $k\ge1$,	respectively, which coincide with the pmf of PRF given in (\ref{prfdist}).
\end{remark}

\begin{remark}
	From (\ref{fprfpgfeq}), the $n$th factorial moment
	of FPRF solves
	\begin{equation}\label{nthfm}
		\frac{\partial^{\nu_1+\nu_2}}{\partial t_1^{\nu_1}\partial t_2^{\nu_2}}\mathbb{E}\prod_{i=0}^{n-1}(\mathscr{N}_{\nu_1,\nu_2}(t_1,t_2)-i)=n^2\lambda\mathbb{E}\prod_{i=0}^{n-2}(\mathscr{N}_{\nu_1,\nu_2}(t_1,t_2)-i),\ n\ge2,
	\end{equation}
	with $\mathbb{E}\prod_{i=0}^{n-1}(\mathscr{N}_{\nu_1,\nu_2}(t_1,t_2)-i)|_{t_1=0,t_2=0}=0$. On solving (\ref{nthfm}) inductively, we get
	\begin{equation*}
		\mathbb{E}\mathscr{N}_{\nu_1,\nu_2}(t_1,t_2)(\mathscr{N}_{\nu_1,\nu_2}(t_1,t_2)-1)\dots(\mathscr{N}_{\nu_1,\nu_2}(t_1,t_2)-n+1)=\frac{(n!)^2(\lambda t_1^{\nu_1}t_2^{\nu_2})^n}{\Gamma(n\nu_1+1)\Gamma(n\nu_2+1)}.
	\end{equation*}
	Thus, the distribution of FPRF can be represented as follows:
	\begin{equation*}
		q_{\nu_1,\nu_2}(k,t_1,t_2)=\sum_{n=k}^{\infty}\frac{(-1)^{n-k}}{k!(n-k)!}\mathbb{E}\mathscr{N}_{\nu_1,\nu_2}(t_1,t_2)(\mathscr{N}_{\nu_1,\nu_2}(t_1,t_2)-1)\dots(\mathscr{N}_{\nu_1,\nu_2}(t_1,t_2)-n+1),\ k\ge0.
	\end{equation*}
\end{remark}

\begin{remark}
		On substituting $\nu_2=1$ in (\ref{c*}), we get a probability distribution $\mathrm{Pr}\{\mathscr{N}_{\nu_1,1}(t_1,t_2)=k\}=(\lambda t_1^{\nu_1}t_2)^kE_{\nu_1,\nu_1 k+1}^{k+1}(-\lambda t_1^{\nu_1}t_2)$, $k\ge0$,		
		where $E_{\alpha,\beta}^{\gamma}(\cdot)$ is the generalized Mittag-Leffler function defined in (\ref{mittag}). Similar to Theorem \ref{diffrelation}, it can be shown that $\mathscr{N}_{\nu_1,1}(t_1,t_2)\overset{d}{=} \mathcal{N}(T_{2\nu_1}(t_1),t_2)$. Here, $\{T_{2\nu_1}(t),\ t\ge0\}$ is a random process whose density is the folded solution of the fractional diffusion equation (\ref{couchyp}) and $\{\mathcal{N}(t_1,t_2),\ (t_1,t_2)\in\mathbb{R}^2_+\}$ is the PRF.
		 Thus, $\{\mathscr{N}_{\nu_1,1}(t_1,t_2),\ (t_1,t_2)\in\mathbb{R}^2_+\}$ gives a different time-changed variant of the PRF. It can be shown that $\mathscr{N}_{\nu_1,1}(t_1,t_2)\overset{d}{=} \mathcal{N}(L_{\nu_1}(t_1),t_2)$, $0<\nu_1<1$, that is, $\{\mathscr{N}_{\nu_1,1}(t_1,t_2),\ (t_1,t_2)\in\mathbb{R}^2_+\}$ is equal in distribution to a PRF on $\mathbb{R}^2_+$ where the first coordinate of time space is random which is governed by an independent inverse $\nu_1$-stable subordinator.  
		 Its mean and variance can be obtained by substituting $\nu_2=1$ in (\ref{fprfmean}) and (\ref{varN}), respectively.		
			Similarly, if we take $\nu_1=1$ then the random field $\{\mathscr{N}_{1,\nu_2}(t_1,t_2),\ (t_1,t_2)\in\mathbb{R}^2_+\}$ agrees the representation given by $\mathscr{N}_{1,\nu_2}(t_1,t_2)\overset{d}{=}\mathcal{N}(t_1,L_{\nu_2}(t_2))$, where $\{L_{\nu_2}(t),\ t\ge0\}$, $0<\nu_2<1$ is an inverse $\nu_2$-stable subordinator independent of the PRF. Also, its distribution is 
				$
				\mathrm{Pr}\{\mathscr{N}_{1,\nu_2}(t_1,t_2)=k\}=(\lambda t_1t_2^{\nu_2})^kE_{\nu_2,\nu_2 k+1}^{k+1}(-\lambda t_1t_2^{\nu_2}),\ k\ge0.
			$
\end{remark}

	Next, we obtain the pgf of FPRF. Let $\{\mathscr{N}^{1-z}_{\nu_1,\nu_2}(t_1,t_2),\ (t_1,t_2)\in\mathbb{R}^2_+\}$, $|z|<1$ be the FPRF with parameter $\lambda(1-z)>0$. Then, on using (\ref{v1}) and (\ref{p1}), we have
	\begin{align*}
		\mathscr{G}_{\nu_1,\nu_2}(z,t_1,t_2)&=\sum_{k=0}^{\infty}z^kq_{\nu_1,\nu_2}(k,t_1,t_2)\\
		&=\sum_{n=0}^{\infty}\frac{n!(-\lambda t_1^{\nu_1}t_2^{\nu_2})^n}{\Gamma(n\nu_1+1)\Gamma(n\nu_2+1)}+\sum_{k=1}^{\infty}\sum_{n=k}^{\infty}z^k\frac{(-1)^{n-k}n_{(n-k)}n_{(k)}(\lambda t_1^{\nu_1}t_2^{\nu_2})^{n}}{\Gamma(n\nu_1+1)\Gamma(n\nu_2+1)}\\
		&=1+\sum_{n=1}^{\infty}\frac{n!(\lambda t_1^{\nu_1}t_2^{\nu_2})^n}{\Gamma(n\nu_1+1)\Gamma(n\nu_2+1)}\left((-1)^n+\sum_{k=1}^{n}\binom{n}{k}z^k(-1)^{n-k}\right)\\
		&=1+\sum_{n=1}^{\infty}\frac{n!(-(1-z)\lambda t_1^{\nu_1}t_2^{\nu_2})^n}{\Gamma(n\nu_1+1)\Gamma(n\nu_2+1)}=\mathrm{Pr}\{\mathscr{N}^{1-z}_{\nu_1,\nu_2}(t_1,t_2)=0\}.
	\end{align*}
	For $\nu_1=\nu_2=1$, it reduces to the pgf of PRF. It is given by  
$
	\mathbb{E}z^{\mathcal{N}(t_1,t_2)}=\exp((z-1)\lambda t_1t_2),\ |z|<1.
$

\begin{remark}
	In view of Theorem 1.1 and Theorem 1.2 of  Baddeley (2009), the FPRF can be characterize by its capacity function $T_{\nu_1,\nu_2}(K)\coloneqq1-\mathrm{Pr}\{\mathscr{N}_{\nu_1,\nu_2}(K)=0\},$ where $K\subset\mathbb{R}^2_+$ is a compact set. For $(t_1,t_2)\in\mathbb{R}^2_+$, it is given by
	\begin{equation*}
		T_{\nu_1,\nu_2}([0,t_1]\times[0,t_2])\coloneqq1-\sum_{n=0}^{\infty}\frac{n!(-\lambda t_1^{\nu_1}t_2^{\nu_2})^n}{\Gamma(n\nu_1+1)\Gamma(n\nu_2+1)}=-\sum_{n=1}^{\infty}\frac{n!(-\lambda t_1^{\nu_1}t_2^{\nu_2})^n}{\Gamma(n\nu_1+1)\Gamma(n\nu_2+1)}.
	\end{equation*}
\end{remark}

Let us consider a sequence of independent and identically distributed (iid) non-negative random variables $Y_1$, $Y_2,\dots,Y_{\mathscr{N}_{\nu_1,\nu_2}(t_1,t_2)}$ with common distribution function $F(\cdot)$. It is assumed that $\mathscr{N}_{\nu_1,\nu_2}(t_1,t_2)$ is independent of $Y_i$'s for all $(t_1,t_2)\in\mathbb{R}^2_+$. The following result provides an explicit expression for the conditional distribution of $k$th order statistics $Y_{(k)}^{\mathscr{N}_{\nu_1,\nu_2}(t_1,t_2)}$, given that the event $\{\mathscr{N}_{\nu_1,\nu_2}(t_1,t_2)\ge k\}$ has occurred.
\begin{theorem}
	Let $\{\mathscr{N}_{\nu_1,\nu_2}(t_1,t_2),\ (t_1,t_2)\in\mathbb{R}^2_+\}$ and $\{\mathscr{N}^{F(v)}_{\nu_1,\nu_2}(t_1,t_2),\ (t_1,t_2)\in\mathbb{R}^2_+\}$ be two fractional Poisson random fields with parameters $\lambda>0$ and $\lambda F(v)>0$, respectively. Then, conditional on the event $\{\mathscr{N}_{\nu_1,\nu_2}(t_1,t_2)\ge k\}$, $k\ge1$, the distribution of $k$th order statistics is given by
	\begin{equation}\label{condord}
		\mathrm{Pr}\left\{Y_{(k)}^{\mathscr{N}_{\nu_1,\nu_2}(t_1,t_2)}\leq v|\mathscr{N}_{\nu_1,\nu_2}(t_1,t_2)\ge k\right\}=\frac{\mathrm{Pr}\{\mathscr{N}^{F(v)}_{\nu_1,\nu_2}(t_1,t_2)\ge k\}}{\mathrm{Pr}\{\mathscr{N}_{\nu_1,\nu_2}(t_1,t_2)\ge k\}}.
	\end{equation}
\end{theorem}
\begin{proof}
	For $(t_1,t_2)\in\mathbb{R}^2_+$ and $k\ge1$, we have
	{\small\begin{align*}
		\mathrm{Pr}&\left\{Y_{(k)}^{\mathscr{N}_{\nu_1,\nu_2}(t_1,t_2)}\leq v|\mathscr{N}_{\nu_1,\nu_2}(t_1,t_2)\ge k\right\}\\
		&=\frac{\sum_{r=k}^{\infty}\mathrm{Pr}\left\{Y_{(k)}^{\mathscr{N}_{\nu_1,\nu_2}(t_1,t_2)}\leq v,\mathscr{N}_{\nu_1,\nu_2}(t_1,t_2)=r\right\}}{\mathrm{Pr}\{\mathscr{N}_{\nu_1,\nu_2}(t_1,t_2)\ge k\}}\\
		&=\frac{1}{\mathrm{Pr}\{\mathscr{N}_{\nu_1,\nu_2}(t_1,t_2)\ge k\}}\sum_{r=k}^{\infty}\mathrm{Pr}\left\{Y_{(k)}^{\mathscr{N}_{\nu_1,\nu_2}(t_1,t_2)}\leq v|\mathscr{N}_{\nu_1,\nu_2}(t_1,t_2)=r\right\}\mathrm{Pr}\{\mathscr{N}_{\nu_1,\nu_2}(t_1,t_2)=r\}\\
		&=\frac{1}{\mathrm{Pr}\{\mathscr{N}_{\nu_1,\nu_2}(t_1,t_2)\ge k\}}\sum_{r=k}^{\infty}\sum_{m=k}^{r}\binom{r}{m}F^m(v)(1-F(v))^{r-m}\sum_{n=r}^{\infty}\frac{(-1)^{n-r}n_{(n-r)}n_{(r)}(\lambda t_1^{\nu_1}t_2^{\nu_2})^{n}}{\Gamma(n\nu_1+1)\Gamma(n\nu_2+1)}\\
		&=\frac{1}{\mathrm{Pr}\{\mathscr{N}_{\nu_1,\nu_2}(t_1,t_2)\ge k\}}\sum_{n=k}^{\infty}\frac{n!(-\lambda t_1^{\nu_1}t_2^{\nu_2})^n}{\Gamma(n\nu_1+1)\Gamma(n\nu_2+1)}\sum_{r=k}^{n}(-1)^r\binom{n}{r} \sum_{m=k}^{r}\binom{r}{m}F^m(v)(1-F(v))^{r-m}\\
		&=\frac{1}{\mathrm{Pr}\{\mathscr{N}_{\nu_1,\nu_2}(t_1,t_2)\ge k\}}\sum_{n=k}^{\infty}\frac{n!(-\lambda t_1^{\nu_1}t_2^{\nu_2})^n}{\Gamma(n\nu_1+1)\Gamma(n\nu_2+1)}\sum_{m=k}^{n}F^m(v)\binom{n}{m} \sum_{r=m}^{n}\binom{n-m}{r-m}(-1)^{r}(1-F(v))^{r-m}\\
		&=\frac{1}{\mathrm{Pr}\{\mathscr{N}_{\nu_1,\nu_2}(t_1,t_2)\ge k\}}\sum_{n=k}^{\infty}\frac{n!(-\lambda t_1^{\nu_1}t_2^{\nu_2})^n}{\Gamma(n\nu_1+1)\Gamma(n\nu_2+1)}F^n(v)\sum_{m=k}^{n}\binom{n}{m}(-1)^{m}\\
		&=\frac{1}{\mathrm{Pr}\{\mathscr{N}_{\nu_1,\nu_2}(t_1,t_2)\ge k\}}\sum_{m=k}^{\infty}\sum_{n=m}^{\infty}\binom{n}{m}\frac{(-1)^{n+m}n!(\lambda t_1^{\nu_1}t_2^{\nu_2}F(v))^n}{\Gamma(n\nu_1+1)\Gamma(n\nu_2+1)}\\
		&=\frac{1}{\mathrm{Pr}\{\mathscr{N}_{\nu_1,\nu_2}(t_1,t_2)\ge k\}}\sum_{m=k}^{\infty}\mathrm{Pr}\{\mathscr{N}^{F(v)}_{\nu_1,\nu_2}(t_1,t_2)=m\},
	\end{align*}}
	where the last step follows from (\ref{p1}). This completes the proof.
\end{proof}
\begin{remark}
	Note that $\lambda F(v)\to\lambda$ as $v\to\infty$. Thus, both sides of (\ref{condord}) converges to one as $v\to\infty$.
\end{remark}

\begin{remark}\label{rem}
	 On taking $k=1$ in (\ref{condord}), we get the conditional distribution of the minimum order statistics $\min\{Y_1,Y_2,\dots,Y_{\mathscr{N}_{\nu_1,\nu_2}(t_1,t_2)}\}$ as follows:
	\begin{equation*}
		\mathrm{Pr}\big\{\min_{1\leq m\leq \mathscr{N}_{\nu_1,\nu_2}(t_1,t_2)}Y_m\leq v|\mathscr{N}_{\nu_1,\nu_2}(t_1,t_2)\ge1\big\}=\frac{1-\mathrm{Pr}\{\mathscr{N}_{\nu_1,\nu_2}^{F(v)}(t_1,t_2)=0\}}{1-\mathrm{Pr}\{\mathscr{N}_{\nu_1,\nu_2}(t_1,t_2)=0\}}.
	\end{equation*}	
	Also, the conditional distribution of the maximum order statistics $\max\{Y_1,Y_2,\dots,Y_{\mathscr{N}_{\nu_1,\nu_2}(t_1,t_2)}\}$ is given by
	\begin{align}
		\mathrm{Pr}\big\{&\max_{1\leq m\leq \mathscr{N}_{\nu_1,\nu_2}(t_1,t_2)}Y_m\leq v|\mathscr{N}_{\nu_1,\nu_2}(t_1,t_2)\ge 1\big\}\nonumber\\
		&=\frac{1}{\mathrm{Pr}\{\mathscr{N}_{\nu_1,\nu_2}(t_1,t_2)\ge 1\}}\sum_{j=1}^{\infty}\mathrm{Pr}\big\{\max\limits_{1\leq m\leq j}Y_m\leq v\big\}\mathrm{Pr}\{\mathscr{N}_{\nu_1,\nu_2}(t_1,t_2)=j\}\nonumber\\
		&=\frac{1}{\mathrm{Pr}\{\mathscr{N}_{\nu_1,\nu_2}(t_1,t_2)\ge 1\}}\sum_{j=1}^{\infty}F^j(v)\sum_{n=j}^{\infty}\frac{(-1)^{n-j}n_{(n-j)}n_{(j)}(\lambda t_1^{\nu_1}t_2^{\nu_2})^{n}}{\Gamma(n\nu_1+1)\Gamma(n\nu_2+1)}\nonumber\\
		&=\frac{1}{\mathrm{Pr}\{\mathscr{N}_{\nu_1,\nu_2}(t_1,t_2)\ge 1\}}\sum_{n=1}^{\infty}\frac{n!(\lambda t_1^{\nu_1}t_2^{\nu_2})^n}{\Gamma(n\nu_1+1)\Gamma(n\nu_2+1)}\sum_{j=1}^{n}\binom{n}{j}(-1)^{n-j}F^j(v)\nonumber\\
		&=\frac{1}{\mathrm{Pr}\{\mathscr{N}_{\nu_1,\nu_2}(t_1,t_2)\ge 1\}}\sum_{n=1}^{\infty}\frac{n!(\lambda t_1^{\nu_1}t_2^{\nu_2})^n}{\Gamma(n\nu_1+1)\Gamma(n\nu_2+1)}((F(v)-1)^n-(-1)^n)\nonumber\\
		&=\frac{\mathrm{Pr}\{\mathscr{N}_{\nu_1,\nu_2}^{1-F(v)}(t_1,t_2)=0\}-\mathrm{Pr}\{\mathscr{N}_{\nu_1,\nu_2}(t_1,t_2)=0\}}{1-\mathrm{Pr}\{\mathscr{N}_{\nu_1,\nu_2}(t_1,t_2)=0\}},\label{max}
	\end{align}
	where $\left\{\mathscr{N}_{\nu_1,\nu_2}^{1-F(v)}(t_1,t_2),\ (t_1,t_2)\in\mathbb{R}^2_+\right\}$ is the FPRF with parameter $\lambda(1-F(v))$. 
	
	In particular, for $\nu_1=\nu_2=1$, (\ref{max}) reduces to
	\begin{equation*}
		\mathrm{Pr}\big\{\max\limits_{1\leq m\leq \mathcal{N}(t_1,t_2)}Y_m\leq v|\mathcal{N}(t_1,t_2)\ge 1\big\}=\frac{e^{-\lambda(1-F(v))t_1t_2}-e^{-\lambda t_1t_2}}{1-e^{-\lambda t_1t_2}},
	\end{equation*}
	which converges to $1$ and $0$ as $v\to\infty$ and $v\to0$, respectively.
	
	Further, if we assume that $\max_{1\leq m\leq \mathscr{N}_{\nu_1,\nu_2}(t_1,t_2)}Y_m=-\infty$ whenever $\mathscr{N}_{\nu_1,\nu_2}(t_1,t_2)=0$ then $\{\mathscr{N}_{\nu_1,\nu_2}(t_1,t_2)=0\}$ $\subset$ $\{\max_{1\leq m\leq \mathscr{N}_{\nu_1,\nu_2}(t_1,t_2)}Y_m\leq v\}$. So,
	\begin{align*}
		\mathrm{Pr}\big\{\max\limits_{1\leq m\leq \mathscr{N}_{\nu_1,\nu_2}(t_1,t_2)}Y_m\leq v\big\}&=\sum_{k=0}^{\infty}\mathrm{Pr}\big\{\max\limits_{1\leq m\leq k}Y_m\leq v\big\}\mathrm{Pr}\{\mathscr{N}_{\nu_1,\nu_2}(t_1,t_2)=k\}\\
		&=\sum_{k=0}^{\infty}\sum_{n=k}^{\infty}F^k(v)\frac{(-1)^{n-k}n_{(n-k)}n_{(k)}(\lambda t_1^{\nu_1}t_2^{\nu_2})^{n}}{\Gamma(n\nu_1+1)\Gamma(n\nu_2+1)}\\
		&=\sum_{n=0}^{\infty}\frac{n!(\lambda t_1^{\nu_1}t_2^{\nu_2})^n}{\Gamma(n\nu_1+1)\Gamma(n\nu_2+1)}\sum_{k=0}^{n}\binom{n}{k}(-1)^{n-k}F^k(v)\\
		&=\mathrm{Pr}\{\mathscr{N}^{1-F(v)}_{\nu_1,\nu_2}(t_1,t_2)=0\}.
	\end{align*}
	Similarly, we take $\min_{1\leq m\leq \mathscr{N}_{\nu_1,\nu_2}(t_1,t_2)}Y_m=\infty$ whenever $\mathscr{N}_{\nu_1,\nu_2}(t_1,t_2)=0$. Then, we have
	$
	\mathrm{Pr}\{\min_{1\leq m\leq \mathscr{N}_{\nu_1,\nu_2}(t_1,t_2)}Y_m> v\}=$ $\mathrm{Pr}\{\mathscr{N}^{F(v)}_{\nu_1,\nu_2}(t_1,t_2)=0\}$, where $\{\mathscr{N}^{F(v)}_{\nu_1,\nu_2}(t_1,t_2),\ (t_1,t_2)\in\mathbb{R}^2_+\}$ is the FPRF with parameter $\lambda F(v)>0$.
\end{remark}
\subsection{A generalization of PRF}\label{AFPRF}
Here, we consider a generalization of the PRF on $\mathbb{R}^2_+$. 
 
 Beghin and Orsingher (2009) gave a different fractional Poisson process $\{\hat{N}_\nu(t),\ t\ge0\}$, $0<\nu\leq1$ whose pmf is given by
$
	\mathrm{Pr}\{\hat{N}_\nu(t)=k\}={(\lambda t)^k}/({\Gamma(k\nu+1)E_{\nu,1}(\lambda t)}),\ k\ge0.
$
 It does not possess the independent increment and the memoryless  properties of the Poisson process. 
 In the similar context, we introduce a generalization of the Poisson random fields on $\mathbb{R}^2_+$.  
 
 Let us consider a random field  $\{\hat{\mathscr{N}}_{\nu_1,\nu_2}(t_1,t_2),\ (t_1,t_2)\in\mathbb{R}^2_+\}$ whose one dimensional distribution is given by
\begin{equation}\label{2nddeffprf}
	\mathrm{Pr}\{\hat{\mathscr{N}}_{\nu_1,\nu_2}(t_1,t_2)=k\}=\frac{(\lambda t_1^{\nu_1}t_2^{\nu_2})^k}{\Gamma(k\nu_1+\nu_2)E_{\nu_1,\nu_2}(\lambda t_1^{\nu_1}t_2^{\nu_2})},\ k\ge0.
\end{equation}
Note that (\ref{2nddeffprf}) is indeed a pmf, that is, it satisfies the regularity conditions. For $\nu_1=\nu_2=1$, the distribution in (\ref{2nddeffprf}) coincide with the distribution of PRF given in (\ref{prfdist}). The mean of $\hat{\mathscr{N}}_{\nu_1,\nu_2}(t_1,t_2)$ is given by 
\begin{align*}
	\mathbb{E}\hat{\mathscr{N}}_{\nu_1,\nu_2}(t_1,t_2)&=\frac{1}{E_{\nu_1,\nu_2}(\lambda t_1^{\nu_1}t_2^{\nu_2})}\sum_{k=0}^{\infty}k\frac{(\lambda t_1^{\nu_1}t_2^{\nu_2})^k}{\Gamma(k\nu_1+\nu_2)}\\
	&=\frac{\lambda}{E_{\nu_1,\nu_2}(\lambda t_1^{\nu_1}t_2^{\nu_2})}\frac{\partial}{\partial\lambda}E_{\nu_1,\nu_2}(\lambda t_1^{\nu_1}t_2^{\nu_2})=\frac{\lambda t_1^{\nu_1}t_2^{\nu_2} E_{\nu_1,\nu_1+\nu_2}^2(\lambda t_1^{\nu_1}t_2^{\nu_2})}{E_{\nu_1,\nu_2}(\lambda t_1^{\nu_1}t_2^{\nu_2})},
\end{align*}
where we have used (\ref{der2mittag}). Here, $E_{\alpha,\beta}^\gamma(\cdot)$ is the generalized Mittag-Leffler function. Also, the variance of $\hat{\mathscr{N}}_{\nu_1,\nu_2}(t_1,t_2)$ is
\begin{align*}
	\mathbb{V}\mathrm{ar}\hat{\mathscr{N}}_{\nu_1,\nu_2}(t_1,t_2)&=\mathbb{E}\hat{\mathscr{N}}_{\nu_1,\nu_2}^2(t_1,t_2)-(\mathbb{E}\hat{\mathscr{N}}_{\nu_1,\nu_2}(t_1,t_2))^2\\
	&=\frac{1}{E_{\nu_1,\nu_2}(\lambda t_1^{\nu_1}t_2^{\nu_2})}\sum_{k=0}^{\infty}k^2\frac{(\lambda t_1^{\nu_1}t_2^{\nu_2})^k}{\Gamma(k\nu_1+\nu_2)}-(\mathbb{E}\hat{\mathscr{N}}_{\nu_1,\nu_2}(t_1,t_2))^2\\
	&=\frac{\lambda}{E_{\nu_1,\nu_2}(\lambda t_1^{\nu_1}t_2^{\nu_2})}\left(\frac{\partial}{\partial\lambda}E_{\nu_1,\nu_2}(\lambda t_1^{\nu_1}t_2^{\nu_2})+\lambda\frac{\partial^2}{\partial\lambda^2}E_{\nu_1,\nu_2}(\lambda t_1^{\nu_1}t_2^{\nu_2})\right)-(\mathbb{E}\hat{\mathscr{N}}_{\nu_1,\nu_2}(t_1,t_2))^2\\
	&=\frac{2(\lambda t_1^{\nu_1}t_2^{\nu_2})^2 E_{\nu_1,\nu_2+2\nu_1}^{3}(\lambda t_1^{\nu_1}t_2^{\nu_2})}{E_{\nu_1,\nu_2}(\lambda t_1^{\nu_1}t_2^{\nu_2})}+\mathbb{E}\hat{\mathscr{N}}_{\nu_1,\nu_2}(t_1,t_2)\left(1-\mathbb{E}\hat{\mathscr{N}}_{\nu_1,\nu_2}(t_1,t_2)\right).
\end{align*}

On taking $\nu_1=\nu_2=1$, we get $	\mathbb{E}\hat{\mathscr{N}}_{1,1}(t_1,t_2)=\mathbb{V}\mathrm{ar}\hat{\mathscr{N}}_{1,1}(t_1,t_2)=\lambda t_1t_2$ which coincide with the mean and variance of the PRF, respectively. Moreover, its pgf is 
\begin{equation*}
	\mathbb{E}z^{\hat{\mathscr{N}}_{\nu_1,\nu_2}(t_1,t_2)}=\sum_{k=0}^{\infty}z^k	\mathrm{Pr}\{\hat{\mathscr{N}}_{\nu_1,\nu_2}(t_1,t_2)=k\}=\frac{E_{\nu_1,\nu_2}(\lambda t_1^{\nu_1}t_2^{\nu_2}z)}{E_{\nu_1,\nu_2}(\lambda t_1^{\nu_1}t_2^{\nu_2})},\ |z|\leq1.
\end{equation*}

\begin{remark}
	Let $Y_1$, $Y_2,\dots$, $Y_{\hat{\mathscr{N}}_{\nu_1,\nu_2}(t_1,t_2)}$ be non-negative random variables with common distribution function $F(\cdot)$. Also, let $Y_i$'s be independent of $\hat{\mathscr{N}}_{\nu_1,\nu_2}(t_1,t_2)$. As done for $\mathscr{N}_{\nu_1,\nu_2}(t_1,t_2)$, the distribution of minimum and maximum order statistics are given as follows:
	\begin{equation*}
	\mathrm{Pr}\big\{\min_{1\leq m\leq \hat{\mathscr{N}}(t_1,t_2)}Y_n> v\big\}=\frac{E_{\nu_1,\nu_2}(\lambda(1-F(v))t_1^{\nu_1}t_2^{\nu_2})}{E_{\nu_1,\nu_2}(\lambda t_1^{\nu_1}t_2^{\nu_2})}
	\end{equation*}
	and
	\begin{equation*}
		\mathrm{Pr}\big\{\max_{1\leq m\leq \hat{\mathscr{N}}(t_1,t_2)}Y_n\leq v\big\}=\frac{E_{\nu_1,\nu_2}(\lambda F(v)t_1^{\nu_1}t_2^{\nu_2})}{E_{\nu_1,\nu_2}(\lambda t_1^{\nu_1}t_2^{\nu_2})}.
	\end{equation*}
\end{remark}
\section{Generalized Poisson random field on $\mathbb{R}^d_+$}\label{sec5} In this section, we introduce and study a generalization of the PRF on finite dimensional Euclidean space by using the generalized Mittag-Leffler function.

Let $\mathcal{B}_{\mathbb{R}_+^d}$ be the Borel sigma algebra on $\mathbb{R}_+^d$, $d\ge1$ and $|B|$ denote the Lebesgue measure of $B\in\mathcal{B}_{\mathbb{R}^d}$. For $|B|<\infty$,
  let us consider a sequence $\{p_{\alpha,\gamma}(k,B),\ k\ge0\}$, $0<\alpha\leq1$, $0<\gamma\leq1$ of real numbers defined as follows:
\begin{equation}\label{genfprfdef}
p_{\alpha,\gamma}(k,B)=\frac{(\gamma)_k(\lambda |B|^\alpha)^k}{k!}E_{\alpha,\alpha k+1}^{\gamma+k}(-\lambda |B|^\alpha),\ k\ge0,
\end{equation}
where $(\gamma)_k=\gamma(\gamma+1)\dots(\gamma+k-1)$ with $(\gamma)_{0}=1$.

Equivalently,
\begin{equation*}
	p_{\alpha,\gamma}(k,B)=\frac{(-\lambda)^k}{k!}\frac{\partial^k}{\partial \lambda^k}E_{\alpha,1}^\gamma(-\lambda |B|^\alpha),
\end{equation*}
which is obtained using (\ref{der2mittag}). It is non-negative due to the complete monotonic characteristic of the generalized Mittag-Leffler function, that is, (see  Mainardi and Garrappa (2015), G\'orska \textit{et al.} (2021))
\begin{equation*}
	(-1)^k\frac{\mathrm{d}^k}{\mathrm{d}x^k}E_{\alpha,1}^\gamma(-x)\ge0
\end{equation*}
for all $k\ge0$ and $x\ge0$. 

On using (\ref{genfprfdef}), we get the following Laplace transform:
\begin{align*}
	\int_{0}^{\infty}e^{-w|B|}\sum_{k=0}^{\infty}p_{\alpha,\gamma}(k,B)\,\mathrm{d}|B|&=\sum_{k=0}^{\infty}\frac{(\gamma)_k\lambda^k}{k!}\int_{0}^{\infty}e^{-w|B|}|B|^{\alpha k}E_{\alpha,\alpha k+1}^{\gamma+k}(-\lambda |B|^\alpha)\,\mathrm{d}|B|\\
	&=\sum_{k=0}^{\infty}\frac{(\gamma)_k\lambda^k}{k!}\frac{w^{\alpha\gamma-1}}{(w^\alpha+\lambda)^{\gamma+k}},\ \ (\text{using (\ref{wetlapmittag})})\\
	&=\frac{w^{\alpha\gamma-1}}{(w^\alpha+\lambda)^\gamma}\sum_{k=0}^{\infty}\frac{(-\gamma)_{(k)}}{k!}\left(-\frac{\lambda }{w^\alpha+\lambda}\right)^k\\
	&=\frac{w^{\alpha\gamma-1}}{(w^\alpha+\lambda )^\gamma}\left(1-\frac{\lambda}{w^\alpha+\lambda }\right)^{-\gamma}=\frac{1}{w},\ w>0,
\end{align*}
where we have used $(\gamma)_k=(-1)^k(-\gamma)_{(k)}$ and the generalized binomial theorem to get the penultimate step. Thus, the sequence defined in (\ref{genfprfdef}) is a valid probability distribution.
In particular, for $\alpha=\gamma=1$, it reduces to the distribution of PRF on $\mathbb{R}^d_+$ given in (\ref{dprfdist}).

Let $\{\mathscr{N}_{\alpha,\gamma}(B),\ B\in\mathcal{B}_{\mathbb{R}^d_+}\}$ be the random field on $\mathbb{R}^d_+$ whose distribution is given by (\ref{genfprfdef}), that is, $\mathrm{Pr}\{\mathscr{N}_{\alpha,\gamma}(B)=k\}=p_{\alpha,\gamma}(k,B)$, $k\ge0$. Then, the Laplace transform of its pgf $G_{\alpha,\gamma}(z,B)\coloneqq\sum_{k=0}^{\infty}z^kp_{\alpha,\gamma}^\delta(k,B)$, $|z|\leq1$ is
	\begin{align*}
		\int_{0}^{\infty}e^{-w|B|}G_{\alpha,\gamma}(z,B)\,\mathrm{d}t_2&=\sum_{k=0}^{\infty}\frac{(\gamma)_k(z\lambda )^k}{k!}\frac{w^{\alpha\gamma-1}}{(w^\alpha+\lambda)^{\gamma+k}}\\
		&=\frac{w^{\alpha\gamma-1}}{(w^\alpha+\lambda)^\gamma}\sum_{k=0}^{\infty}\frac{(-\gamma)_{(k)}}{k!}\left(-\frac{z\lambda}{w^\alpha+\lambda}\right)^k\\
		&=\frac{w^{\alpha\gamma-1}}{(w^\alpha+(1-z)\lambda)^\gamma},\ w>0,
	\end{align*}
	where we have used (\ref{wetlapmittag}). 
	Its inverse Laplace transform yields $G_{\alpha,\gamma}(z,B)=E_{\alpha,1}^\gamma((z-1)\lambda|B|^\alpha).$
	
	 The $n$th factorial moment of $\mathscr{N}_{\alpha,\gamma}(B)$ is given by
\begin{equation*}
	\mathbb{E}\mathscr{N}_{\alpha,\gamma}(B)(\mathscr{N}_{\alpha,\gamma}(B)-1)\dots(\mathscr{N}_{\alpha,\gamma}(B)-n+1)=\frac{\partial^n}{\partial z^n}G_{\alpha,\gamma}(z,B)\bigg|_{z=1}=\frac{(\gamma)_n(\lambda |B|^\alpha)^n}{\Gamma(n\alpha+1)},\ n\ge1.
\end{equation*}
Hence, its distribution has the following representation:
\begin{align*}
	p_{\alpha,\gamma}(k,B)&=\frac{(\gamma)_k(\lambda |B|^\alpha)^k}{k!}\sum_{n=0}^{\infty}\frac{(\gamma+k)_n(-\lambda|B|^\alpha)^n}{\Gamma(n\alpha+k\alpha+1)n!}=\sum_{n=k}^{\infty}\frac{(-1)^{n-k}}{(n-k)!k!}\mathbb{E}\prod_{i=0}^{n-1}(\mathscr{N}_{\alpha,\gamma}(B)-i),\ k\ge0,
\end{align*}
where we have used $(\gamma)_k(\gamma+k)_{n-k}=(\gamma)_n$.
\begin{remark}
	Suppose $Y_1$, $Y_2,\dots$, $Y_{{\mathscr{N}}_{\alpha,\gamma}(B)}$ be non-negative iid random variables with common distribution function $F(\cdot)$. Let $\{\mathscr{N}_{\alpha,\gamma}(B),\ B\in\mathcal{B}_{\mathbb{R}^d}\}$ be independent of $Y_i$'s. Then, the distribution of minimum and maximum order statistics are 
	$
		\mathrm{Pr}\left\{\min_{1\leq m\leq{\mathscr{N}_{\alpha,\gamma}}(B)}Y_n> v\right\}=E_{\alpha,1}^\gamma(-F(v)\lambda|B|^\alpha)$  and $
		\mathrm{Pr}\left\{\max_{1\leq m\leq {\mathscr{N}_{\alpha,\gamma}}(B)}Y_n\leq v\right\}=E_{\alpha,1}^\gamma((F(v)-1)|B|^\alpha)
	$, respectively. This can be obtained along the similar lines to that of Remark \ref{rem}.
\end{remark}
\begin{remark}
For $d=2$, (\ref{genfprfdef}) is the distribution of a generalized version of the PRF on $\mathbb{R}^2_+$. It is given by 
\begin{equation*}
	p_{\alpha,\gamma}(k,t_1,t_2)=\frac{(\gamma)_k(\lambda (t_1t_2)^\alpha)^k}{k!}E_{\alpha,\alpha k+1}^{\gamma+k}(-\lambda (t_1t_2)^\alpha),\ k\ge0.
	\end{equation*}
	For $\gamma=\alpha=1$, it further reduces to the distribution of PRF on $\mathbb{R}^2_+$ given in (\ref{prfdist}).
\end{remark}
\subsection{Generalized Poisson process}
Note that for $d=1$ and $t\in\mathbb{R}_+$,  (\ref{genfprfdef}) gives a family of distributions given by
\begin{equation}\label{genfpp}
	\mathrm{Pr}\{\mathscr{N}_{\alpha,\gamma}((0,t))=k\}=\frac{(\gamma)_k(\lambda t^\alpha)^k}{k!}E_{\alpha,\alpha k+1}^{\gamma+k}(-\lambda t^\alpha),\ k\ge0.
\end{equation}
For $\gamma=1$, (\ref{genfpp}) reduces to the distribution of fractional Poisson process studied by Beghin and Orsingher (2010). Moreover, for $\gamma=\alpha=1$, it gives the distribution of homogeneous Poisson process with parameter $\lambda>0$.

Let $\mathscr{N}_{\alpha,\gamma}(t)=\mathscr{N}_{\alpha,\gamma}((0,t))$, $t\ge0$. Consider the process $\{\mathscr{N}_{\alpha,\gamma}(t),\ t\ge0\}$, $0<\alpha\leq1$, $0<\gamma\leq1$ whose pmf $p_{\alpha,\gamma}(k,t)=\mathrm{Pr}\{\mathscr{N}_{\alpha,\gamma}(t)=k\}$, $k\ge0$ is given by (\ref{genfpp}). We call this process as the generalized Poisson process (GPP). Its pgf is given by 
\begin{equation}\label{genpppgf}
	\mathscr{G}_{\alpha,\gamma}(z,t)\coloneqq\mathbb{E}z^{\mathscr{N}_{\alpha,\gamma}(t)}=E_{\alpha,1}^\gamma((z-1)\lambda t^\alpha),\ |z|\leq1.
\end{equation}
The mean and variance of GPP are $\mathbb{E}\mathscr{N}_{\alpha,\gamma}(t)=\gamma\lambda t^\alpha/\Gamma(\alpha+1)$ and
\begin{equation*}
	\mathbb{V}\mathrm{ar}\mathscr{N}_{\alpha,\gamma}(t)=\frac{\gamma(\gamma+1)(\lambda t^\alpha)^2}{\Gamma(2\alpha+1)}+\frac{\gamma\lambda t^\alpha}{\Gamma(\alpha+1)}\left(1-\frac{\gamma\lambda t^\alpha}{\Gamma(\alpha+1)}\right),
\end{equation*}
respectively.
\begin{remark}
	The distribution of the first waiting time $\mathscr{T}_{\alpha,\gamma}=\inf\{t>0:\mathscr{N}_{\alpha,\gamma}(t)=1\}$ of GPP is given by
	$\mathrm{Pr}\{\mathscr{T}_{\alpha,\gamma}>t\}=E_{\alpha,1}^\gamma(-\lambda t^\alpha)$, $t>0$. Thus, $\mathscr{T}_{\alpha,\gamma}$ is a generalized Mittag-Leffler random variable. In particular, for $\gamma=1$, it becomes the Mittag-Leffler random variable. For $\alpha=\gamma=1$, it further reduces to the well known exponential random variable. Moreover, for $t\ge0$, $s\ge0$, we have
	\begin{equation*}
		\mathrm{Pr}\{\mathscr{T}_{\alpha,\gamma}>t+s|\mathscr{T}_{\alpha,\gamma}>s\}=\frac{E_{\alpha,1}^\gamma(-\lambda(t+s)^\alpha)}{E_{\alpha,1}^\gamma(-\lambda s^\alpha)}.
	\end{equation*}
	Thus, the GPP does not constitute the lack of memory property of the Poisson process.
\end{remark}
The following result connects the GPP with Poisson process via a random process, that is, we show that GPP is equal in distribution to a time-changed Poisson process.
\begin{theorem}\label{c1}
	Let $\{\psi_{\alpha,\gamma}(t),\ t\ge0\}$, $0<\alpha\leq1$, $0<\gamma\leq1$ be a continuous-time random process such that the Laplace transform of its density function with respect to time variable is
	\begin{equation}\label{genpplap}
		\int_{0}^{\infty}e^{-wt}\mathrm{Pr}\{\psi_{\alpha,\gamma}(t)\in\mathrm{d}x\}\,\mathrm{d}t=\frac{w^{\alpha\gamma-1}x^{\gamma-1}}{\Gamma(\gamma)}e^{-w^\alpha x}\,\mathrm{d}x,\ w>0.
	\end{equation}
	Then, the following time-changed relationship holds:
	\begin{equation*}
		\mathscr{N}_{\alpha,\gamma}(t)\overset{d}{=}\mathscr{N}(\psi_{\alpha,\gamma}(t)),\ t\ge0,
	\end{equation*}
	where $\{\mathscr{N}(t),\ t\ge0\}$ is a Poisson process with parameter $\lambda>0$ and it is independent of the process $\{\psi_{\alpha,\gamma}(t),\ t\ge0\}$.
\end{theorem} 
\begin{proof}
	For $t\ge0$, the pgf of $\mathscr{N}(\psi_{\alpha,\gamma}(t))$ is given by
	\begin{equation}\label{genpplap1}
	\mathbb{E}z^{\mathscr{N}(\psi_{\alpha,\gamma}(t))}=\int_{0}^{\infty}\mathbb{E}z^{\mathscr{N}(x)}\mathrm{Pr}\{\psi_{\alpha,\gamma}(t)\in\mathrm{d}x\}=\int_{0}^{\infty}e^{\lambda x(z-1)}\mathrm{Pr}\{\psi_{\alpha,\gamma}(t)\in\mathrm{d}x\},\ |z|\leq1,
	\end{equation}
	where we have used $\mathbb{E}z^{\mathscr{N}(x)}=e^{\lambda x(z-1)}$, the pgf of Poisson process.	On taking the Laplace transform on both sides of (\ref{genpplap1}) and using (\ref{genpplap}), we get
	\begin{equation}\label{genpplap2}
		\int_{0}^{\infty}e^{-wt}\mathbb{E}z^{\mathscr{N}(\psi_{\alpha,\gamma}(t))}\,\mathrm{d}t=\int_{0}^{\infty}e^{\lambda x(z-1)}\frac{w^{\alpha\gamma-1}x^{\gamma-1}}{\Gamma(\gamma)}e^{-w^\alpha x}\,\mathrm{d}x=\frac{w^{\alpha\gamma-1}}{(w^\alpha+(1-z)\lambda)^\gamma}.
	\end{equation}
	On using (\ref{wetlapmittag}), the  inverse Laplace transform of (\ref{genpplap2}) gives (\ref{genpppgf}). The uniqueness of pgf completes the proof.
\end{proof}
\begin{remark}
	The Laplace transform of the process $\{\psi_{\alpha,\gamma}(t),\ t\ge0\}$ is
	\begin{equation}\label{denlap1}
		\mathbb{E}e^{-u\psi_{\alpha,\gamma}(t)}=\int_{0}^{\infty}e^{-ux}\mathrm{Pr}\{\psi_{\alpha,\gamma}(t)\in\mathrm{d}x\},\ u>0.
	\end{equation}
	On taking the Laplace transform on both sides of (\ref{denlap1}) and using (\ref{genpplap}), we get
	\begin{equation*}
		\int_{0}^{\infty}e^{-wt}\mathbb{E}e^{-u\psi_{\alpha,\gamma}(t)}\,\mathrm{d}t=\int_{0}^{\infty}e^{-ux}\frac{w^{\alpha\gamma-1}x^{\gamma-1}}{\Gamma(\gamma)}e^{-w^\alpha x}\,\mathrm{d}x=\frac{w^{\alpha\gamma-1}}{(w^\alpha+u)^\gamma},
	\end{equation*}
	whose inversion yields $\mathbb{E}e^{-u\psi_{\alpha,\gamma}(t)}=E_{\alpha,1}^\gamma(-ut^\alpha)$, $u>0$.
	
	For $c>0$, we have
	\begin{equation*}
		\mathbb{E}e^{-uc^{-\alpha}\psi_{\alpha,\gamma}(ct)}=E_{\alpha,1}^\gamma(-ut^\alpha)=\mathbb{E}e^{-u\psi_{\alpha,\gamma}(t)},\ t\ge0.
	\end{equation*}
So, $c^\alpha\psi_{\alpha,\gamma}(t)\overset{d}{=}\psi_{\alpha,\gamma}(ct)$.	Hence, $\{\psi_{\alpha,\gamma}(t),\ t\ge0\}$ is a self similar process with exponent $\alpha>0$.
	
	On using (\ref{genpplap}), the Laplace transform of the mean of $\psi_{\alpha,\gamma}(t)$ can be written as:
	\begin{equation*}
	\int_{0}^{\infty}e^{-wt}\mathbb{E}\psi_{\alpha,\gamma}(t)\,\mathrm{d}t=\frac{w^{\alpha\gamma-1}}{\Gamma(\gamma)}\int_{0}^{\infty}x^\gamma e^{-w^\alpha x}\,\mathrm{d}x=\frac{\Gamma(\gamma+1)}{\Gamma(\gamma)w^{\alpha+1}}=\frac{\gamma}{w^{\alpha+1}}.
	\end{equation*}
	Its inversion yields $\mathbb{E}\psi_{\alpha,\gamma}(t)=\gamma t^\alpha/\Gamma(\alpha+1)$. Similarly, we have $\mathbb{E}\psi_{\alpha,\gamma}^2(t)=(\gamma+1)\gamma t^{2\alpha}/\Gamma(2\alpha+1).$ So, 
	\begin{equation*}
		\mathbb{V}\mathrm{ar}\psi_{\alpha,\gamma}(t)=\left(\frac{\gamma+1}{\Gamma(2\alpha+1)}-\frac{\gamma}{\Gamma^2{(\alpha+1)}}\right)\gamma t^{2\alpha},\ t\ge0.
	\end{equation*}
\end{remark}
\begin{remark}\label{c0}
	On taking $\gamma=1$ in (\ref{genpplap}), we get $\int_{0}^{\infty}e^{-wt}\mathrm{Pr}\{\psi_{\alpha,1}(t)\in\mathrm{d}x\}\,\mathrm{d}t={w^{\alpha-1}}e^{-w^\alpha x}\,\mathrm{d}x$. Hence, in view of (\ref{fold}), we have $\{\psi_{\alpha,1}(t),\ t\ge0\}\overset{d}{=}\{T_{2\alpha}(t),\ t\ge0\}$, where $\{T_{2\alpha}(t),\ t\ge0\}$ is a random process whose density is the folded solution of (\ref{couchyp}) for $\nu_1=\alpha$. For  $\alpha\in(0,1)$, it follows from (\ref{subdenlap1}) that $\{\psi_{\alpha,1}(t),\ t\ge0\}$ and inverse $\alpha$-stable subordinator are equal in distribution. Also, $\{\psi_{1/2,1}(t),\ t\ge0\}$ has same distribution to that of reflecting Brownian motion.
\end{remark}

Next, we consider a different time-changed variants of the PRF on $\mathbb{R}^2_+$.
Let $\{\psi_{\alpha_i,\gamma_i},\ t\ge0\}$, $0<\alpha_i\leq1$, $0<\nu_i\leq1$, $i=1,2$ be two independent random processes with Laplace transform $\mathbb{E}e^{-u\psi_{\alpha_i,\gamma_i}}=E_{\alpha_i,1}^{\gamma_i}(-ut^{\alpha_i})$, $u>0$. Let $\{\mathcal{N}(t_1,t_2),\ (t_1,t_2)\in\mathbb{R}^2_+\}$ be a PRF which is independent of $\{\psi_{\alpha_1,\gamma_1}(t), t\ge0\}$ and $\{\psi_{\alpha_2,\gamma_2}(t), t\ge0\}$. Consider the following time-changed PRF:
\begin{equation*}
	\mathscr{N}_{\alpha_1,\gamma_1}^{\alpha_2,\gamma_2}(t_1,t_2)\coloneqq\mathcal{N}(\psi_{\alpha_1,\gamma_1}(t_1),\psi_{\alpha_2,\gamma_2}(t_2)),\ (t_1,t_2)\in\mathbb{R}^2_+.
\end{equation*}
So,
\begin{equation}\label{atcprf}
	\mathrm{Pr}\{\mathscr{N}_{\alpha_1,\gamma_1}^{\alpha_2,\gamma_2}(t_1,t_2)=k\}=\frac{\lambda^k}{k!}\int_{0}^{\infty}\int_{0}^{\infty}(x_1x_2)^ke^{-\lambda x_1x_2}\prod_{i=1}^{2}\mathrm{Pr}\{\psi_{\alpha_i,\gamma_i}(t_i)\in\mathrm{d}x_i\},\ k\ge0.
\end{equation}

On taking the Laplace transforms on both sides of (\ref{atcprf}) with respect to $t_1$ and $t_2$ and using (\ref{genpplap}), we get
\begin{align}
\int_{0}^{\infty}\int_{0}^{\infty}&e^{-(w_1t_1+w_2t_2)}\mathrm{Pr}\{\mathscr{N}_{\alpha_1,\gamma_1}^{\alpha_2,\gamma_2}(t_1,t_2)=k\}\,\mathrm{d}t_1\,\mathrm{d}t_2\nonumber\\
&=\frac{\lambda^kw_1^{\alpha_1\gamma_1-1}w_2^{\alpha_2\gamma_2-1}}{k!\Gamma(\gamma_1)\Gamma(\gamma_2)}\int_{0}^{\infty}\int_{0}^{\infty}x_2^{k+\gamma_2-1}x_1^{k+\gamma_1-1}e^{-(w_1^{\alpha_1}+\lambda x_2)x_1}e^{-w_2^{\alpha_2}x_2}\,\mathrm{d}x_1\,\mathrm{d}x_2\nonumber\\
&=\frac{\Gamma(k+\gamma_1)\lambda^kw_2^{\alpha_2\gamma_2-1}}{k!\Gamma(\gamma_1)\Gamma(\gamma_2)}\int_{0}^{\infty}x_2^{k+\gamma_2-1}e^{-w_2^{\alpha_2}x_2}\frac{w_1^{\alpha_1\gamma_1-1}}{(w_1^{\alpha_1}+\lambda x_2)^{k+\gamma_1}}\,\mathrm{d}x_2.\label{2fprf}
\end{align}
On taking the inverse Laplace transform on both sides of (\ref{2fprf}) with respect to $w_1$, we have
\begin{align*}
\int_{0}^{\infty}&e^{-w_2t_2}\mathrm{Pr}\{\mathscr{N}_{\alpha_1,\gamma_1}^{\alpha_2,\gamma_2}(t_1,t_2)=k\}\,\mathrm{d}t_2\\
&=	\frac{\Gamma(k+\gamma_1)\lambda^kw_2^{\alpha_2\gamma_2-1}t_1^{\alpha_1k}}{k!\Gamma(\gamma_1)\Gamma(\gamma_2)}\int_{0}^{\infty}E_{\alpha_1,\alpha_1k+1}^{k+\gamma_1}(-\lambda x_2t_1^{\alpha_1})x_2^{k+\gamma_2-1}e^{-w_2^{\alpha_2}x_2}\,\mathrm{d}x_2\\
&=\frac{\Gamma(k+\gamma_1)\lambda^kw_2^{\alpha_2\gamma_2-1}t_1^{\alpha_1k}}{k!\Gamma(\gamma_1)\Gamma(\gamma_2)}\sum_{n=0}^{\infty}\frac{\Gamma(n+k+\gamma_1)(-\lambda t_1^{\alpha_1})^n}{\Gamma(k+\gamma_1)\Gamma(n\alpha_1+\alpha_1k+1)n!}\int_{0}^{\infty}x_2^{n+k+\gamma_2-1}e^{-w_2^{\alpha_2}x_2}\,\mathrm{d}x_2\\
&=\frac{(\lambda t_1^{\alpha_1})^k}{k!\Gamma(\gamma_1)\Gamma(\gamma_2)}\sum_{n=0}^{\infty}\frac{\Gamma(n+k+\gamma_1)\Gamma(n+k+\gamma_2)(-\lambda t_1^{\alpha_1})^n}{\Gamma(n\alpha_1+\alpha_1k+1)n!w_2^{\alpha_2(n+k)+1}}.
\end{align*}
Further, on taking the inverse Laplace transform with  respect to $w_2$, we get
\begin{align}
	\mathrm{Pr}\{\mathscr{N}_{\alpha_1,\gamma_1}^{\alpha_2,\gamma_2}(t_1,t_2)=k\}&=\frac{(\lambda t_1^{\alpha_1}t_2^{\alpha_2})^k}{k!\Gamma(\gamma_1)\Gamma(\gamma_2)}\sum_{n=0}^{\infty}\frac{\Gamma(n+k+\gamma_1)\Gamma(n+k+\gamma_2)(-\lambda t_1^{\alpha_1}t_2^{\alpha_2})^n}{\Gamma(n\alpha_1+\alpha_1k+1)\Gamma(n\alpha_2+k\alpha_2+1)n!}\nonumber\\
	&=\frac{(\lambda t_1^{\alpha_1}t_2^{\alpha_2})^k}{k!\Gamma(\gamma_1)\Gamma(\gamma_2)}{}_2\Psi_2\left[\begin{matrix}
		(k+\gamma_1,1),&(k+\gamma_2,1)\\\\
		(\alpha_1k+1,\alpha_1),&(\alpha_2k+1,\alpha_2)
	\end{matrix}\Bigg|-\lambda t_1^{\alpha_1}t_2^{\alpha_2}\right],\ k\ge0.\label{c**}
\end{align}

For $\gamma_1=\gamma_2=1$, (\ref{c**}) reduces to (\ref{c*}). Therefore, $\{\mathscr{N}_{\alpha_1,1}^{\alpha_2,1}(t_1,t_2),\ (t_1,t_2)\in\mathbb{R}^2_+\}\overset{d}{=}\{\mathscr{N}_{\alpha_1,\alpha_2}(t_1,t_2),\ (t_1,t_2)\in\mathbb{R}^2_+\}$.

 The Laplace transforms of the mean of $\mathscr{N}_{\alpha_1,\gamma_1}^{\alpha_2,\gamma_2}(t_1,t_2)$ is given by
{\small\begin{align*}
	\int_{0}^{\infty}\int_{0}^{\infty}e^{-(w_1t_1+w_2t_2)}\mathbb{E}\mathscr{N}_{\alpha_1,\gamma_1}^{\alpha_2,\gamma_2}(t_1,t_2)\,\mathrm{d}t_1\,\mathrm{d}t_2=\lambda \prod_{i=1}^{2}w_i^{\alpha_i\gamma_i-1}\int_{0}^{\infty}x_i^{\gamma_i}e^{-w_i^{\alpha_i}x_i}\,\mathrm{d}x_i=\frac{\lambda\Gamma(\gamma_1+1)\Gamma(\gamma_2+1)}{w_1^{\alpha_1+1}w_2^{\alpha_2+1}},
\end{align*}}
whose inversion yields
\begin{equation*}
	\mathbb{E}\mathscr{N}_{\alpha_1,\gamma_1}^{\alpha_2,\gamma_2}(t_1,t_2)=\frac{\lambda\Gamma(\gamma_1+1)\Gamma(\gamma_2+1)t_1^{\alpha_1}t_2^{\alpha_2}}{\Gamma(\alpha_1+1)\Gamma(\alpha_2+1)}.
\end{equation*}

Similarly, we have
\begin{align*}
	\int_{0}^{\infty}\int_{0}^{\infty}e^{-(w_1t_1+w_2t_2)}\mathbb{E}(\mathscr{N}_{\alpha_1,\gamma_1}^{\alpha_2,\gamma_2}(t_1,t_2))^2\,\mathrm{d}t_1\,\mathrm{d}t_2
	&=\sum_{r=0}^{1}\lambda^{r+1} \prod_{i=1}^{2}w_i^{\alpha_i\gamma_i-1}\int_{0}^{\infty}x_i^{\gamma_i+r}e^{-w_i^{\alpha_i}x_i}\,\mathrm{d}x_i\\
	&=\frac{\lambda\Gamma(\gamma_1+1)\Gamma(\gamma_2+1)}{w_1^{\alpha_1+1}w_2^{\alpha_2+1}}+\frac{\lambda^2\Gamma(\gamma_1+2)\Gamma(\gamma_2+2)}{w_1^{2\alpha_1+1}w_2^{2\alpha_2+1}},
\end{align*}
and its inversion gives
\begin{equation*}
	\mathbb{E}(\mathscr{N}_{\alpha_1,\gamma_1}^{\alpha_2,\gamma_2}(t_1,t_2))^2=\frac{\lambda\Gamma(\gamma_1+1)\Gamma(\gamma_2+1)t_1^{\alpha_1}t_2^{\alpha_2}}{\Gamma(\alpha_1+1)\Gamma(\alpha_2+1)}+\frac{\lambda^2\Gamma(\gamma_1+2)\Gamma(\gamma_2+2)t_1^{2\alpha_1}t_2^{2\alpha_2}}{\Gamma(2\alpha_1+1)\Gamma(2\alpha_2+1)}.
\end{equation*}
Hence, the variance of $\{\mathscr{N}_{\alpha_1,\gamma_1}^{\alpha_2,\gamma_2}(t_1,t_2),\ (t_1,t_2)\in\mathbb{R}^2_+\}$ is
{\footnotesize\begin{equation*}
	\mathbb{V}\mathrm{ar}\mathscr{N}_{\alpha_1,\gamma_1}^{\alpha_2,\gamma_2}(t_1,t_2)=	\frac{\lambda\Gamma(\gamma_1+1)\Gamma(\gamma_2+1)t_1^{\alpha_1}t_2^{\alpha_2}}{\Gamma(\alpha_1+1)\Gamma(\alpha_2+1)}+\left(\frac{\Gamma(\gamma_1+2)\Gamma(\gamma_2+2)}{\Gamma(2\alpha_1+1)\Gamma(2\alpha_2+1)}-\frac{\Gamma^2(\gamma_1+1)\Gamma^2(\gamma_2+1)}{\Gamma^2(\alpha_1+1)\Gamma^2(\alpha_2+1)}\right)(\lambda t_1^{\alpha_1}t_2^{\alpha_2})^2.
\end{equation*}}
\begin{remark}
		The pgf of $\mathscr{N}_{\alpha_1,\gamma_1}^{\alpha_2,\gamma_2}(t_1,t_2)$ is
		\begin{align*}
			\mathscr{G}_{\alpha_1,\gamma_1}^{\alpha_2,\gamma_2}(z,t_1,t_2)&=\sum_{k=0}^{\infty}z^k\frac{(\lambda t_1^{\alpha_1}t_2^{\alpha_2})^k}{k!\Gamma(\gamma_1)\Gamma(\gamma_2)}\sum_{n=0}^{\infty}\frac{\Gamma(n+k+\gamma_1)\Gamma(n+k+\gamma_2)(-\lambda t_1^{\alpha_1}t_2^{\alpha_2})^n}{\Gamma(n\alpha_1+\alpha_1k+1)\Gamma(n\alpha_2+k\alpha_2+1)n!}\\
			&=\frac{1}{\Gamma(\gamma_1)\Gamma(\gamma_2)}\sum_{k=0}^{\infty}\frac{(-z)^k}{k!}\sum_{n=k}^{\infty}\frac{\Gamma(n+\gamma_1)\Gamma(n+\gamma_2)(-\lambda t_1^{\alpha_1}t_2^{\alpha_2})^n}{\Gamma(n\alpha_1+1)\Gamma(n\alpha_2+1)(n-k)!}\\
			&=\frac{1}{\Gamma(\gamma_1)\Gamma(\gamma_2)}\sum_{n=0}^{\infty}\frac{\Gamma(n+\gamma_1)\Gamma(n+\gamma_2)(-\lambda t_1^{\alpha_1}t_2^{\alpha_2})^n}{\Gamma(n\alpha_1+1)\Gamma(n\alpha_2+1)n!}\sum_{k=0}^{n}\binom{n}{k}(-z)^k\\
			&=\frac{1}{\Gamma(\gamma_1)\Gamma(\gamma_2)}{}_2\Psi_2\left[\begin{matrix}
				(\gamma_1,1),&(\gamma_2,1)\\\\
				(1,\alpha_1),&(1,\alpha_2)
			\end{matrix}\Bigg|\lambda(z-1) t_1^{\alpha_1}t_2^{\alpha_2}\right],\ |z|\leq1.
		\end{align*}
		So, its $n$th factorial moment is given by
		\begin{equation*}
			\mathbb{E}\prod_{i=0}^{n-1}(\mathscr{N}_{\alpha_1,\gamma_1}^{\alpha_2,\gamma_2}(t_1,t_2)-i)=\frac{\partial^n}{\partial z^n}\mathscr{G}_{\alpha_1,\gamma_1}^{\alpha_2,\gamma_2}(z,t_1,t_2)\bigg|_{z=1}=\frac{\Gamma(n+\gamma_1)\Gamma(n+\gamma_2)(\lambda t_1^{\alpha_1}t_2^{\alpha_2})^n}{\Gamma(\gamma_1)\Gamma(\gamma_2)\Gamma(n\alpha_1+1)\Gamma(n\alpha_2+1)}.
		\end{equation*}
		Hence, the pmf of $\mathscr{N}_{\alpha_1,\gamma_1}^{\alpha_2,\gamma_2}(t_1,t_2)$ has the following representation:
		\begin{equation*}
			\mathrm{Pr}\{\mathscr{N}_{\alpha_1,\gamma_1}^{\alpha_2,\gamma_2}(t_1,t_2)=k\}=\sum_{n=k}^{\infty}\frac{(-1)^{n-k}}{k!(n-k)}\mathbb{E}\prod_{i=0}^{n-1}(\mathscr{N}_{\alpha_1,\gamma_1}^{\alpha_2,\gamma_2}(t_1,t_2)-i),\ k\ge0.
		\end{equation*}
\end{remark}
\section{Random motions governed by GPP}\label{sec6}
Beghin and Orsingher (2009) studied the planer random motion of a particle moving with finite velocity at random time and changing the direction of its motion on every fractional Poisson arrivals. They observed that if the particle changes its direction for sufficiently large number of times then the distribution of its position vector is more concentrated. Moreover, the sample paths of the resulting planer motion is coiled around the origin. Here, we study  time-changed random motions of a particle on $\mathbb{R}$ and $\mathbb{R}^2$ whose change of direction is governed by a GPP. 

First, we consider the linear random motion of a particle that changes its direction at each GPP arrivals.
\subsection{Linear random motion driven by GPP} Let us consider a particle starting from origin and moving on real line back and forth with finite velocity $v>0$. It changes its direction at every Poisson arrival. Let $\{\mathscr{N}(t),\ t\ge0\}$ be a Poisson process with rate $\lambda>0$, and let $X(t)$ denotes the random position of particle at time $t$. Then, $X(t)\in[-vt,vt]$ with probability one. For $\eta\in\mathbb{R}$, the characteristics function of $X(t)$ is given by (see Ratanov and Kolesnik (2022))
\begin{equation}\label{1rmcf}
	\mathbb{E}e^{i\eta X(t)}=e^{-\lambda t}\left(\cosh(t\sqrt{\eta^2v^2-\lambda^2})+\frac{\lambda}{\sqrt{\eta^2v^2-\lambda^2}}\sinh(t\sqrt{\eta^2v^2-\lambda^2})\right),\ t\ge0.
\end{equation}

Let $\psi_{\alpha,\gamma}(t)$ be independent of the position of particle $X(t)$ for all $t\ge0$ where the process $\{\psi_{\alpha,\gamma}(t),\ t\ge0\}$ is defined in Theorem \ref{c1}. Then, the characteristics function of a time-changed process $\{X(\psi_{\alpha,\gamma}(t)),\ t\ge0\}$ is given by
{\small\begin{align*}
	\mathbb{E}e^{i\eta X(\psi_{\alpha,\gamma}(t))}&=\int_{0}^{\infty}e^{i\eta X(u)}\mathrm{Pr}\{\psi_{\alpha,\gamma}(t)\in\mathrm{d}u\}\\
	&=\int_{0}^{\infty}e^{-\lambda u}\bigg(\cosh(u\sqrt{\eta^2v^2-\lambda^2})+\frac{\lambda}{\sqrt{\eta^2v^2-\lambda^2}}\sinh(u\sqrt{\eta^2v^2-\lambda^2})\bigg)\mathrm{Pr}\{\psi_{\alpha,\gamma}(t)\in\mathrm{d}u\}.
\end{align*}}
On using (\ref{genpplap}), we get 
{\small\begin{align}
\int_{0}^{\infty}e^{-wt}\mathbb{E}e^{i\eta X(\psi_{\alpha,\gamma}(t))}\,\mathrm{d}t&=\int_{0}^{\infty}e^{-\lambda u}\bigg(\cosh(u\sqrt{\eta^2v^2-\lambda^2})+\frac{\lambda}{\sqrt{\eta^2v^2-\lambda^2}}\sinh(u\sqrt{\eta^2v^2-\lambda^2})\bigg)\nonumber\\
&\ \ \cdot \frac{w^{\alpha\gamma-1}}{\Gamma(\gamma)}u^{\gamma-1}e^{-w^\alpha u}\,\mathrm{d}u\nonumber\\
&=\frac{w^{\alpha\gamma-1}}{2}\bigg((w^\alpha+\lambda-\sqrt{\eta^2v^2-\lambda^2})^{-\gamma}+(w^\alpha+\lambda+\sqrt{\eta^2v^2-\lambda^2})^{-\gamma}\nonumber\\
&\ \ +\frac{\lambda}{\sqrt{\eta^2v^2-\lambda^2}}\bigg((w^\alpha+\lambda-\sqrt{\eta^2v^2-\lambda^2})^{-\gamma}-(w^\alpha+\lambda+\sqrt{\eta^2v^2-\lambda^2})^{-\gamma}\bigg)\bigg),\label{1rmlap}
\end{align}}
where we have used the following results (see Gradshteyn and Ryzhik (2007), p. 1113):
\begin{equation*}
	\int_{0}^{\infty}e^{-wx}x^{\gamma-1}\sinh(ax)\,\mathrm{d}x=\frac{\Gamma(\gamma)}{2}((w-a)^{-\gamma}-(w+a)^{-\gamma}),\ w>|a|,\ \gamma>-1\\
\end{equation*}
and
\begin{equation*}
	\int_{0}^{\infty}e^{-wx}x^{\gamma-1}\cosh(ax)\,\mathrm{d}x=\frac{\Gamma(\gamma)}{2}((w-a)^{-\gamma}+(w+a)^{-\gamma}),\ w>|a|, \gamma>0
\end{equation*}
to obtain the last step.

On taking the inverse Laplace transform of (\ref{1rmlap}) and using (\ref{wetlapmittag}), we obtain
\begin{align*}
	\mathbb{E}e^{i\eta X(\psi_{\alpha,\gamma}(t))}&=2^{-1}\left(E_{\alpha,1}^\gamma(-(\lambda-\sqrt{\eta^2v^2-\lambda^2})t^\alpha)+E_{\alpha,1}^\gamma(-(\lambda+\sqrt{\eta^2v^2-\lambda^2})t^\alpha)\right)\\
	&\ \ +\frac{\lambda}{2\sqrt{\eta^2v^2-\lambda^2}}\left(E_{\alpha,1}^\gamma(-(\lambda-\sqrt{\eta^2v^2-\lambda^2})t^\alpha)-E_{\alpha,1}^\gamma(-(\lambda+\sqrt{\eta^2v^2-\lambda^2})t^\alpha)\right).
\end{align*}
\begin{remark}
	For $\gamma=1$ and $\alpha=1/2$, on using (\ref{Mittaglimit}), we get the following asymptotic behaviour of the characteristics function of $X(\psi_{1/2,1}(t))$  for large $t$:
	\begin{align*}
		\mathbb{E}e^{i\eta X(\psi_{1/2,1}(t))}&\sim\left(\exp((\lambda-\sqrt{\eta^2v^2-\lambda^2})^2t)+\exp((\lambda+\sqrt{\eta^2v^2-\lambda^2})^2t)\right)\\
		&\ \ +\frac{\lambda}{\sqrt{\eta^2v^2-\lambda^2}}\left(\exp((\lambda-\sqrt{\eta^2v^2-\lambda^2})^2t)-\exp((\lambda+\sqrt{\eta^2v^2-\lambda^2})^2t)\right)\\
		&=\left(\exp((\eta^2v^2-2\lambda\sqrt{\eta^2v^2-\lambda^2})t)+\exp((\eta^2v^2+2\lambda\sqrt{\eta^2v^2-\lambda^2})t)\right)\\
		&\ \ +\frac{\lambda}{\sqrt{\eta^2v^2-\lambda^2}}\left(\exp((\eta^2v^2-2\lambda\sqrt{\eta^2v^2-\lambda^2})t)-\exp((\eta^2v^2+2\lambda\sqrt{\eta^2v^2-\lambda^2})t)\right)\\
		&=2e^{\eta^2v^2t}\left(\cosh(2t\lambda\sqrt{\eta^2v^2-\lambda^2})+\frac{\lambda}{\sqrt{\eta^2v^2-\lambda^2}}\sinh(2t\lambda\sqrt{\eta^2v^2-\lambda^2})\right).
	\end{align*}
\end{remark}
\begin{remark}
	For $\alpha=\gamma=1$, we have
	\begin{align*}
		\mathbb{E}e^{i\eta X(\psi_{1,1}(t))}&=2^{-1}\left(\exp(-(\lambda-\sqrt{\eta^2v^2-\lambda^2})t)+\exp(-(\lambda+\sqrt{\eta^2v^2-\lambda^2})t)\right)\\
		&\ \ +\frac{\lambda}{2\sqrt{\eta^2v^2-\lambda^2}}\left(\exp(-(\lambda-\sqrt{\eta^2v^2-\lambda^2})t)-\exp(-(\lambda+\sqrt{\eta^2v^2-\lambda^2})t)\right)\\
		&=e^{-\lambda t}\left(\cosh(t\sqrt{\eta^2v^2-\lambda^2})+\frac{\lambda}{\sqrt{\eta^2v^2-\lambda^2}}\sinh(t\sqrt{\eta^2v^2-\lambda^2})\right),
	\end{align*}
	which coincide with (\ref{1rmcf}). 
\end{remark}
\subsection{Planer random motion governed by GPP} Let us consider a planer random motion of a particle starting from origin and moving with finite constant velocity $v>0$. Its direction changes on Poisson arrivals with rate $\lambda>0$, that is, at every Poisson arrivals the  particle takes a new direction with uniform distribution on $\{\textbf{z}=(x,y)\in\mathbb{R}^2:||\textbf{z}||=x^2+y^2=1\}$. Let $\mathrm{d}\textbf{z}=(\mathrm{d}x,\mathrm{d}y)$ be the infinitesimal element of $\mathbb{R}^2$ with Lebesgue measure $|\mathrm{d}\textbf{z}|=\mathrm{d}x\mathrm{d}y$ and $\textbf{Z}(t)=(X(t),Y(t))$ denote the position of particle at time $t$, and let $\{\mathscr{N}(t),\ t\ge0\}$ be the Poisson process. Then, at any time $t$ the particle is located in the region $B_{vt}=\{\textbf{z}\in\mathbb{R}^2:||\textbf{z}||\leq vt\}$ with probability one and the conditional characteristics function of $\textbf{Z}(t)$ is given as follows (see Kolesnik and Orsingher (2005)):
\begin{equation}\label{rmcf}
	\mathbb{E}\{e^{i\boldsymbol{\delta}\textbf{Z}'(t)}|\mathscr{N}(t)=k\}=\int\int_{B_{vt}}e^{i\boldsymbol{\delta}\textbf{z}'}\mathrm{Pr}\{\textbf{Z}(t)\in\mathrm{d}\textbf{z}|\mathscr{N}(t)=k\}=\frac{2^{\frac{k}{2}}\Gamma(\frac{k}{2}+1)J_{\frac{k}{2}}(vt||\boldsymbol{\delta}||)}{(vt||\boldsymbol{\delta}||)^{\frac{k}{2}}},\ k\ge0,
\end{equation}
where $\boldsymbol{\delta}=(\delta_1,\delta_2)\in\mathbb{R}^2$ and $\textbf{z}'$ denotes the transpose of $\textbf{z}$. Here, $J_{\frac{k}{2}}(\cdot)$ is the Bessel function of first kind of order ${k}/{2}$. Also, for $t>0$ and $\textbf{z}\in B _{vt}$, the corresponding conditional distribution is given by
\begin{equation}\label{rmdist}
	\mathrm{Pr}\{\textbf{Z}(t)\in\mathrm{d}\textbf{z}|\mathscr{N}(t)=k\}=\begin{cases}
		\delta\left(v^2t^2-||\textbf{z}||^2\right),\ k=0,\vspace{0.15cm}\\
		\frac{k}{2\pi (vt)^k}\left(v^2t^2-||\textbf{z}||^2\right)^{\frac{k}{2}-1}|\mathrm{d}\textbf{z}|,\ k\ge1,
	\end{cases}
\end{equation}
where $\delta(\cdot)$ is the Dirac delta function.

Here, we obtain the distribution of position of particle at random times where the time variable is represented by the process $\{\psi_{\alpha,\gamma}(t),\ t\ge0\}$. First, we derive the conditional characteristics function of $\textbf{Z}_{\alpha,\gamma}(t)=\textbf{Z}(\psi_{\alpha,\gamma}(t)),\ t>0$ given the event $\{\mathscr{N}_{\alpha,\gamma}(t)=k\}$ for some $ k\ge1$.

\begin{theorem}
	For $k\ge1$, the conditional characteristics function of the vector $\textbf{Z}_{\alpha,\gamma}(t)=(X(\psi_{\alpha,\gamma}(t)), Y(\psi_{\alpha,\gamma}(t)))$ is given by
	\begin{align}\label{frmcf}
		\mathbb{E}\{e^{i\boldsymbol{\delta}\textbf{Z}_{\alpha,\gamma}'(t)}|\mathscr{N}_{\alpha,\gamma}(t)=k\}&=\frac{1}{B(\frac{1}{2},\frac{k+1}{2})\Gamma(\gamma)}\int_{0}^{1}r^{-\frac{1}{2}}(1-r)^{\frac{k-1}{2}}\nonumber\\
		&\hspace{3cm} \cdot{}_1\Psi_2\left[\begin{matrix}
			(\gamma,2)\\\\
			(1,2),\ (1,2\alpha)
		\end{matrix}\Bigg| -v^2||\boldsymbol{\delta}||^2t^{2\alpha}r\right]\,\mathrm{d}r,\ t\ge0,
	\end{align}
	where ${}_1\Psi_2(\cdot)$ is the generalized Wright function.
\end{theorem}
\begin{proof}
For $t>0$, $k\ge0$ and $\boldsymbol{\delta}\in\mathbb{R}^2$, we have
\begin{align}
	\mathbb{E}\{e^{i\boldsymbol{\delta}\textbf{Z}_{\alpha,\gamma}'(t)}|\mathscr{N}_{\alpha,\gamma}(t)=k\}&=\int_{0}^{\infty}\mathbb{E}\{e^{i\boldsymbol{\delta}\textbf{Z}'(\psi_{\alpha,\gamma}(t))}|\mathscr{N}(\psi_{\alpha,\gamma}(t))=k,\psi_{\alpha,\gamma}(t)=u\}\mathrm{Pr}\{\psi_{\alpha,\gamma}(t)\in\mathrm{d}u\}\nonumber\\
	&=\int_{0}^{\infty}\mathbb{E}\{e^{i\boldsymbol{\delta}\textbf{Z}'(u)}|\mathscr{N}(u)=k\}\mathrm{Pr}\{\psi_{\alpha,\gamma}(t)\in\mathrm{d}u\}\nonumber\\
	&=\int_{0}^{\infty}\frac{2^{\frac{k}{2}}\Gamma(\frac{k}{2}+1)J_{\frac{k}{2}}(vu||\boldsymbol{\delta}||)}{(vu||\boldsymbol{\delta}||)^{\frac{k}{2}}}\mathrm{Pr}\{\psi_{\alpha,\gamma}(t)\in\mathrm{d}u\},\label{Rmlap1}
\end{align}
where we have used (\ref{rmcf}) to get the last step. On taking the Laplace transform on both sides of (\ref{Rmlap1}) and using (\ref{genpplap}), we get
\begin{align*}
	\int_{0}^{\infty}e^{-wt}	\mathbb{E}\{e^{i\boldsymbol{\delta}\textbf{Z}_{\alpha,\gamma}'(t)}|\mathscr{N}_{\alpha,\gamma}(t)=k\}\,\mathrm{d}t&=\int_{0}^{\infty}\frac{2^{\frac{k}{2}}\Gamma(\frac{k}{2}+1)J_{\frac{k}{2}}(vu||\boldsymbol{\delta}||)}{(vu||\boldsymbol{\delta}||)^{\frac{k}{2}}}\frac{w^{\alpha\gamma-1}u^{\gamma-1}}{\Gamma(\gamma)}e^{-w^\alpha u}\,\mathrm{d}u\\
	&=\frac{\Gamma(\frac{k}{2}+1)w^{\alpha\gamma-1}}{\Gamma(\gamma)}\sum_{n=0}^{\infty}\frac{(-1)^n(v||\boldsymbol{\delta}||/2)^{2n}}{n!\Gamma(n+\frac{k}{2}+1)}\int_{0}^{\infty}u^{2n+\gamma-1}e^{-w^\alpha u}\,\mathrm{d}u\\
	&=\frac{\Gamma(\frac{k}{2}+1)w^{\alpha\gamma-1}}{\Gamma(\gamma)}\sum_{n=0}^{\infty}\frac{(-1)^n(v||\boldsymbol{\delta}||/2)^{2n}\Gamma(2n+\gamma)}{n!\Gamma(n+\frac{k}{2}+1)w^{\alpha(2n+\gamma)}}\\
	&=\frac{\Gamma(\frac{k}{2}+1)}{\Gamma(\gamma)}\sum_{n=0}^{\infty}\frac{(-1)^n(v||\boldsymbol{\delta}||/2)^{2n}\Gamma(2n+\gamma)}{n!\Gamma(n+\frac{k}{2}+1)w^{2n\alpha+1}}.
\end{align*}
Its inverse Laplace transform yields
\begin{align}
	\mathbb{E}\{e^{i\boldsymbol{\delta}\textbf{Z}_{\alpha,\gamma}'(t)}|\mathscr{N}_{\alpha,\gamma}&(t)=k\}\nonumber\\
	&=\frac{\Gamma(\frac{k}{2}+1)}{\Gamma(\gamma)}\sum_{n=0}^{\infty}\frac{(-1)^n(v||\boldsymbol{\delta}||/2)^{2n}\Gamma(2n+\gamma)t^{2n\alpha}}{n!\Gamma(n+\frac{k}{2}+1)\Gamma(2n\alpha+1)}\label{genwrtform}\\
	&=\frac{\Gamma(\frac{k}{2}+1)}{\Gamma(\frac{k+1}{2})\Gamma(\gamma)}\sum_{n=0}^{\infty}\frac{\Gamma(\frac{k+1}{2})\Gamma(n+\frac{1}{2})}{\Gamma(n+\frac{k}{2}+1)\Gamma(n+\frac{1}{2})}\frac{\Gamma(2n+\gamma)(-(v||\boldsymbol{\delta}||t^\alpha/2)^2)^n}{n!\Gamma(2n\alpha+1)}\nonumber\\
	&=\frac{\Gamma(\frac{k}{2}+1)}{\Gamma(\frac{k+1}{2})\Gamma(\gamma)}\int_{0}^{1}r^{-\frac{1}{2}}(1-r)^{\frac{k-1}{2}}\sum_{n=0}^{\infty}\frac{\Gamma(2n+\gamma)(-v^2||\boldsymbol{\delta}||^2t^{2\alpha} r/4)^n}{n!\Gamma(2n\alpha+1)\Gamma(n+\frac{1}{2})}\,\mathrm{d}r\nonumber\\
	&=\frac{1}{B(\frac{1}{2},\frac{k+1}{2})\Gamma(\gamma)}\int_{0}^{1}r^{-\frac{1}{2}}(1-r)^{\frac{k-1}{2}}\sum_{n=0}^{\infty}\frac{\Gamma(2n+\gamma)\Gamma(n)(-v^2||\boldsymbol{\delta}||^2t^{2\alpha} r/4)^n}{n!\Gamma(2n\alpha+1)2^{1-2n}\Gamma(2n)}\,\mathrm{d}r\nonumber\\
	&=\frac{1}{B(\frac{1}{2},\frac{k+1}{2})\Gamma(\gamma)}\int_{0}^{1}r^{-\frac{1}{2}}(1-r)^{\frac{k-1}{2}}\sum_{n=0}^{\infty}\frac{\Gamma(2n+\gamma)(-v^2||\boldsymbol{\delta}||^2t^{2\alpha} r)^n}{\Gamma(2n\alpha+1)\Gamma(2n+1)}\,\mathrm{d}r.\nonumber
\end{align}
This completes the proof.
\end{proof}
\begin{remark}
	From (\ref{genwrtform}), the conditional characteristics function have the following representation:
	\begin{equation}\label{genwritcf}
		\mathbb{E}\{e^{i\boldsymbol{\delta}\textbf{Z}_{\alpha,\gamma}'(t)}|\mathscr{N}_{\alpha,\gamma}(t)=k\}=\frac{\Gamma(\frac{k}{2}+1)}{\Gamma(\gamma)}{}_1\Psi_2\left[\begin{matrix}
			(\gamma,2)\\\\
			(\frac{k}{2}+1,1),\ (1,2\alpha)
		\end{matrix}\Bigg| -\frac{v^2||\boldsymbol{\delta}||^2t^{2\alpha}}{4}\right],\ t\ge0.
	\end{equation}
	For $\alpha=\gamma=1$, it further reduces to (\ref{rmcf}) as follows:
	\begin{align*}
		\mathbb{E}\{e^{i\boldsymbol{\delta}\textbf{Z}_{1,1}'(t)}|\mathscr{N}_{1,1}(t)=k\}&={\Gamma\left(\frac{k}{2}+1\right)}{}_1\Psi_2\left[\begin{matrix}
			(1,2)\\\\
			(\frac{k}{2}+1,1),\ (1,2)
		\end{matrix}\Bigg| -\frac{v^2||\boldsymbol{\delta}||^2t^{2}}{4}\right]\\
		&={\Gamma\left(\frac{k}{2}+1\right)}\sum_{n=0}^{\infty}\frac{(-1)^n(v||\boldsymbol{\delta}||/2)^{2n}\Gamma(2n+1)t^{2n}}{n!\Gamma(n+\frac{k}{2}+1)\Gamma(2n+1)}\\
		&=\frac{2^{\frac{k}{2}}\Gamma(\frac{k}{2}+1)}{(v||\boldsymbol{\delta}||t)^{\frac{k}{2}}}\sum_{n=0}^{\infty}\frac{(-1)^n(v||\boldsymbol{\delta}||t/2)^{2n+\frac{k}{2}}}{n!\Gamma(n+\frac{k}{2}+1)}=\frac{2^{\frac{k}{2}}\Gamma(\frac{k}{2}+1)J_{\frac{k}{2}}(v||\boldsymbol{\delta}||t)}{(v||\boldsymbol{\delta}||t)^{\frac{k}{2}}}.
	\end{align*}	
	Moreover, on taking $\gamma=1$ in (\ref{frmcf}), we get
	\begin{align*}
		\mathbb{E}\{e^{i\boldsymbol{\delta}\textbf{Z}_{\alpha,1}'(t)}|\mathscr{N}_{\alpha,1}(t)=k\}&=\frac{1}{B(\frac{1}{2},\frac{k+1}{2})}\int_{0}^{1}r^{-\frac{1}{2}}(1-r)^{\frac{k-1}{2}}{}_1\Psi_2\left[\begin{matrix}
			(1,2)\\\\
			(1,2),\ (1,2\alpha)
		\end{matrix}\Bigg| -v^2||\boldsymbol{\delta}||^2t^{2\alpha}r\right]\,\mathrm{d}r\\
		&=\frac{1}{B(\frac{1}{2},\frac{k+1}{2})}\int_{0}^{1}r^{-\frac{1}{2}}(1-r)^{\frac{k-1}{2}}E_{2\alpha,1}(-v^2||\boldsymbol{\delta}||^2t^{2\alpha}r)\,\mathrm{d}r,
	\end{align*}
	which agrees with the result obtained in Eq. (3.15) of Beghin and Orsingher (2009).
\end{remark}

The following result provides the conditional distribution of $\textbf{Z}_{\alpha,\gamma}(t)$.
\begin{theorem}\label{thm6.2}
For $k\ge0$, conditional on the event $\{\mathscr{N}_{\alpha,\gamma}(t)=k\}$, the distribution of the random vector $\textbf{Z}_{\alpha,\gamma}(t)$ is given by
{\small\begin{align}\label{frmdist}
	\mathrm{Pr}\{\textbf{Z}_{\alpha,\gamma}(t)\in\mathrm{d}\textbf{z}|\mathscr{N}_{\alpha,\gamma}(t)=k\}/|\mathrm{d}\textbf{z}|&=\frac{\Gamma(\frac{k}{2}+1)}{2\pi\Gamma(\gamma)}\int_{0}^{\infty}rJ_0(r||\textbf{z}||){}_1\Psi_2\left[\begin{matrix}
		(\gamma,2)\\\\
		(\frac{k}{2}+1,1),\ (1,2\alpha)
	\end{matrix}\Bigg| -\frac{v^2r^2t^{2\alpha}}{4}\right]\,\mathrm{d}r,
\end{align}}
where $J_0(\cdot)$ is the Bessel function of order zero.
\end{theorem}
\begin{proof}
On taking the inverse Fourier transform on both sides of (\ref{genwritcf}), we get
\begin{align*}
	\mathrm{Pr}&\{\textbf{Z}_{\alpha,\gamma}(t)\in\mathrm{d}\textbf{z}|\mathscr{N}_{\alpha,\gamma}(t)=k\}/|\mathrm{d}\textbf{z}|\\
	&=\frac{\Gamma(\frac{k}{2}+1)}{(2\pi)^2\Gamma(\gamma)}\int_{-\infty}^{\infty}\int_{-\infty}^{\infty}e^{-i\boldsymbol{\delta}\textbf{z}'}{}_1\Psi_2\left[\begin{matrix}
		(\gamma,2)\\\\
		(\frac{k}{2}+1,1),\ (1,2\alpha)
	\end{matrix}\Bigg| -\frac{v^2||\boldsymbol{\delta}||^2t^{2\alpha}}{4}\right]|\mathrm{d}\boldsymbol{\delta}|\\
	&=\frac{\Gamma(\frac{k}{2}+1)}{(2\pi)^2\Gamma(\gamma)}\int_{0}^{\infty}\left(\int_{0}^{2\pi}e^{-ir(x\cos\theta+y\sin\theta)}\,\mathrm{d}\theta\right){}_1\Psi_2\left[\begin{matrix}
		(\gamma,2)\\\\
		(\frac{k}{2}+1,1),\ (1,2\alpha)
	\end{matrix}\Bigg| -\frac{v^2r^2t^{2\alpha}}{4}\right]r\,\mathrm{d}r.
\end{align*}
On using the following integral representation of the Bessel function (see Gradshteyn and Ryzhik (2007))
\begin{equation*}
	J_0(r\sqrt{x^2+y^2})=\frac{1}{2\pi}\int_{0}^{2\pi}e^{-ir(x\cos\theta+y\sin\theta)}\,\mathrm{d}\theta,
\end{equation*}
we get the required result. 
\end{proof}
\begin{remark}
	On using  (\ref{frmcf}) and following similar lines to the proof of Theorem \ref{thm6.2}, we get the following representation of the  conditional distribution of $\textbf{Z}_{\alpha,\gamma}(t)$:
	\begin{align}
		\mathrm{Pr}\{\textbf{Z}_{\alpha,\gamma}(t)\in\mathrm{d}\textbf{z}|\mathscr{N}_{\alpha,\gamma}(t)=k\}/|\mathrm{d}\textbf{z}|&=\frac{1}{2\pi B(\frac{1}{2},\frac{k+1}{2})\Gamma(\gamma)}\int_{0}^{\infty}rJ_0(r||\textbf{z}||)\int_{0}^{1}y^{-\frac{1}{2}}(1-y)^{\frac{k-1}{2}}\nonumber\\
		&\hspace{2cm} \cdot{}_1\Psi_2\left[\begin{matrix}
			(\gamma,2)\\\\
			(1,2),\ (1,2\alpha)
		\end{matrix}\Bigg| -v^2r^2t^{2\alpha}y\right]\,\mathrm{d}y\,\mathrm{d}r.\label{frmdist1}
	\end{align}
	For $\gamma=1$, it reduces to
	\begin{align*}
		\mathrm{Pr}\{\textbf{Z}_{\alpha,1}&(t)\in\mathrm{d}\textbf{z}|\mathscr{N}_{\alpha,1}(t)=k\}/|\mathrm{d}\textbf{z}|\\
		&=\frac{1}{2\pi B(\frac{1}{2},\frac{k+1}{2})}\int_{0}^{\infty}rJ_0(r||\textbf{z}||)\int_{0}^{1}u^{-\frac{1}{2}}(1-u)^{\frac{k-1}{2}}{}_1\Psi_2\left[\begin{matrix}
			(1,2)\\\\
			(1,2),\ (1,2\alpha)
		\end{matrix}\Bigg| -v^2r^2t^{2\alpha}u\right]\,\mathrm{d}u\,\mathrm{d}r\\		
		&=\frac{1}{2\pi B(\frac{1}{2},\frac{k+1}{2})}\int_{0}^{\infty}rJ_0(r||\textbf{z}||)\int_{0}^{1}u^{-\frac{1}{2}}(1-u)^{\frac{k-1}{2}} E_{2\alpha,1}(-v^2r^2t^{2\alpha}u)\,\mathrm{d}u\,\mathrm{d}r.
	\end{align*}	
	Further, for $\gamma=1$ and $\alpha={1}/{2}$, (\ref{frmdist1}) reduces to 
	\begin{align*}
		\mathrm{Pr}\{\textbf{Z}_{\frac{1}{2},1}(t)&\in\mathrm{d}\textbf{z}|\mathscr{N}_{\frac{1}{2},1}(t)=k\}/|\mathrm{d}\textbf{z}|\\
		&=\frac{1}{2\pi B(\frac{1}{2},\frac{k+1}{2})}\int_{0}^{\infty}rJ_0(r||\textbf{z}||)\int_{0}^{1}u^{-\frac{1}{2}}(1-u)^{\frac{k-1}{2}}\exp(-v^2r^2tu)\,\mathrm{d}u\,\mathrm{d}r\\
		&=\frac{1}{2\pi B(\frac{1}{2},\frac{k+1}{2})}\int_{0}^{1}u^{-\frac{1}{2}}(1-u)^{\frac{k-1}{2}}\sum_{n=0}^{\infty}\frac{(-1)^n(||\textbf{z}||/2)^{2n}}{n!\Gamma(n+1)}\int_{0}^{\infty}r^{2n+1}\exp(-v^2r^2tu)\,\mathrm{d}r\,\mathrm{d}u\\
		&=\frac{1}{2\pi B(\frac{1}{2},\frac{k+1}{2})}\int_{0}^{1}u^{-\frac{1}{2}}(1-u)^{\frac{k-1}{2}}\sum_{n=0}^{\infty}\frac{(-1)^n(||\textbf{z}||/2)^{2n}}{n!\Gamma(n+1)}\frac{\Gamma(n+1)}{2(v^2tu)^n}\\
		&=\frac{1}{4\pi B(\frac{1}{2},\frac{k+1}{2})}\int_{0}^{1}u^{-\frac{1}{2}}(1-u)^{\frac{k-1}{2}}\frac{\exp\left(-{||\textbf{z}||^2}/{4v^2tu}\right)}{v^2tu}\,\mathrm{d}u,
	\end{align*}
	which agrees with the result obtained in Eq. (3.3) of Beghin and Orsingher (2009). Thus, $$\mathrm{Pr}\left\{\textbf{Z}_{\frac{1}{2},1}(t)\in\mathrm{d}\textbf{z}|\mathscr{N}_{\frac{1}{2},1}(t)=k\right\}=\mathrm{Pr}\{\textbf{Z}(|B(t)|)\in\mathrm{d}\textbf{z}|\mathscr{N}(|B(t)|)=k\},\ t\ge0,$$
	where $\{|B(t)|,\ t\ge0\}$ is the reflecting Brownian motion.
\end{remark}

\section{Compound fractional Poisson random field}\label{sec7}
The compound Poisson process extends the concept of Poisson process by allowing for random jump sizes instead of fixed unit step size jumps. It has applications across diverse fields such as insurance (see Dickson (2016)), reliability analysis (see Wang (2022)), mining (see Kruczek \textit{et al.} (2018)), evolutionary biology (see Huelsenbeck \textit{et al.} (2000)), \textit{etc}. For various types of the fractional generalizations of the compound Poisson process, we refer the reader to Beghin and Macci ((2014), (2016)), Di Crescenzo \textit{et al.} (2015), Khandakar and Kataria (2023) and the references therein.
Here, we consider a fractional version of the compound Poisson random field. 

Let $\{Z_i\}_{i\ge0}$ be a collection of iid random variables, and let $\{\mathscr{N}_{\nu_1,\nu_2}(t_1,t_2),\ (t_1,t_2)\in\mathbb{R}^2_+\}$, $\nu_1,\nu_2\in(0,1]$ be the FPRF  as defined in Section \ref{sec3}. Further, we assume that the random variables $Z_i$'s are independent of the FPRF. Let us consider the random process $\{\mathcal{Y}_{\nu_1,\nu_2}(t_1,t_2),\ (t_1,t_2)\in\mathbb{R}^2_+\}$ on positive plane $\mathbb{R}^2_+$ defined as follows:
\begin{equation}\label{cFPRF}
	\mathcal{Y}_{\nu_1,\nu_2}(t_1,t_2)\coloneqq\sum_{i=1}^{\mathscr{N}_{\nu_1,\nu_2}(t_1,t_2)}Z_i.
\end{equation}
We call this process as the compound fractional Poisson random field (CFPRF). Moreover, if $\nu_1=\nu_2=1$ then CFPRF is the compound Poisson random field (CPRF). Note that CFPRF can also be considered as a random walk subordinated to the FPRF.

Let $F_Z(\cdot)$ be the distribution function of $Z_1$. Then, for $(t_1,t_2)\in\mathbb{R}^2_+$ such that $\mathscr{N}_{\nu_1,\nu_2}(t_1,t_2)<\infty$, the distribution function $F_{\mathcal{Y}_{\nu_1,\nu_2}}(y)=\mathrm{Pr}\{\mathcal{Y}_{\nu_1,\nu_2}(t_1,t_2)\leq y\}$, $y\in\mathbb{R}$ of CFPRF is given by
	\begin{align*}
		F_{\mathcal{Y}_{\nu_1,\nu_2}}(y,t_1,t_2)
		&=\sum_{k=0}^{\infty}\mathrm{Pr}\bigg\{\sum_{i=1}^{k}Z_i\leq y\bigg\}q_{\nu_1,\nu_2}(k,t_1,t_2)\nonumber\\
		&=q_{\nu_1,\nu_2}(0,t_1,t_2)F_Z^{*0}(y)+\sum_{k=1}^{\infty}q_{\nu_1,\nu_2}(k,t_1,t_2) F_Z^{*k}(y)\nonumber\\
		&=\mathcal{H}(y)+\sum_{n=1}^{\infty}\frac{n!(-\lambda t_1^{\nu_1}t_2^{\nu_2})^n}{\Gamma(n\nu_1+1)\Gamma(n\nu_2+1)}\bigg(\mathcal{H}(y)+\sum_{k=1}^{n}\binom{n}{k}(-1)^{k}F_Z^{*k}(y)\bigg),
	\end{align*}
	where we have used (\ref{v1}) and (\ref{p1}) to get the last step. Here, $F_Z^{*k}(y)$ denotes the $k$-fold convolution of $F_Z(\cdot)$ and $\mathcal{H}(y)$ is the Heaviside function defined by
	\begin{equation*}
		\mathcal{H}(y)=\begin{cases}
			1,\ y\ge0,\\
			0,\ y<0.
		\end{cases}
	\end{equation*} 	
	Thus, the density function of CFPRF is given by
	\begin{equation*}
		f_{\mathcal{Y}_{\nu_1,\nu_2}}(y,t_1,t_2)=\delta(y)+\sum_{n=1}^{\infty}\frac{n!(-\lambda t_1^{\nu_1}t_2^{\nu_2})^n}{\Gamma(n\nu_1+1)\Gamma(n\nu_2+1)}\bigg(\delta(y)+\sum_{k=1}^{n}\binom{n}{k}(-1)^{k}f_Z^{*k}(y)\bigg),\ y\in\mathbb{R},
	\end{equation*}
	where $f_Z^{*k}(y)$ is $k$-fold convolution of the probability density function $f_Z(z)=\mathrm{d}F_Z(z)/\mathrm{d}z$.
	Moreover, if we take $\nu_1=\nu_2=1$ then we get the density function of the compound Poisson random field in the following form:
	\begin{equation*}
		f_{\mathcal{Y}_{\nu_1,\nu_2}}(y,t_1,t_2)=\delta(y)e^{-\lambda t_1t_2}+\sum_{k=1}^{\infty}\frac{(\lambda t_1 t_2)^k}{k!}f_Z^{*k}(y),\ y\in\mathbb{R}.
	\end{equation*}

In the following theorem, we show that the similar result holds for any spatial point process. Its proof is similar to the case of CFPRF. Hence, it is omitted.
\begin{theorem}
	Let $\{Z_i\}_{i\ge0}$ be a sequence of iid real valued random variables with distribution function $F_Z(\cdot)$, and let $\{\mathcal{N}(B),\ B\in\mathcal{B}_{\mathbb{R}^d}\}$ be a spatial point process on $\mathbb{R}^d$. Also, assume that $\{Z_i\}_{i\ge0}$ is independent of $\{\mathcal{N}(B),\ B\in\mathcal{B}_{\mathbb{R}^d}\}$. For $B\in\mathcal{B}_{\mathbb{R}^d}$ such that $\mathcal{N}(B)<\infty$, consider the process
	$
	\mathcal{Y}(B)=\sum_{i=1}^{\mathcal{N}(B)}Z_i.
	$
	Then, the cumulative distribution function of $\mathcal{Y}(B)$ is given by
	\begin{equation*}
		\mathrm{Pr}\{\mathcal{Y}(B)\leq y\}=\mathrm{Pr}\{\mathcal{N}(B)=0\}\mathcal{H}(y)+\sum_{k=1}^{\infty}\mathrm{Pr}\{\mathcal{N}(B)=k\}F_Z^{*k}(y).
	\end{equation*}	
\end{theorem}
\subsection{Generalized CPRF} Here, we define a generalized compound Poisson random field on $\mathbb{R}^d_+$ by using the fractional random field defined by (\ref{genfprfdef}). Let us consider a CPRF $\{{\mathcal{Y}}_{\alpha,\gamma}(B),\ B\in\mathcal{B}_{\mathbb{R}^d_+}\}$ defined as follows:
\begin{equation*}
		{\mathcal{Y}}_{\alpha,\gamma}(B)\coloneqq\sum_{i=1}^{{\mathscr{N}}_{\alpha,\gamma}(B)}Z_i,
\end{equation*}
where $\{\mathscr{N}_{\alpha,\gamma}(B),\ B\in\mathcal{B}_{\mathbb{R}^d}\}$ is the random field whose distribution is given by (\ref{genfprfdef}). In particular, for $d=\gamma=1$, it reduces to the compound fractional Poisson process (see Scalas (2011)). Moreover, for $\gamma=\alpha=1$, it reduces to the compound Poisson process.

	Let $\phi(u)=\mathbb{E}e^{iuZ_1}$, $u\in\mathbb{R}$ be the characteristics function of $Z_1$. Then, the Laplace transform of the characteristics function of  ${\mathcal{Y}}_{\alpha,\gamma}(B)$ is
\begin{align*}
	\int_{0}^{\infty}e^{-w|B|}\mathbb{E}e^{iu\mathcal{Y}_{\alpha,\gamma}(B)}\,\mathrm{d}|B|&=\int_{0}^{\infty}e^{-w|B|}\sum_{k=0}^{\infty}\phi(u)^k\frac{(\gamma)_k(\lambda |B|^\alpha)^k}{k!}E_{\alpha,\alpha k+1}^{\gamma+k}(-\lambda |B|^\alpha)\,\mathrm{d}|B|\\
	&=\sum_{k=0}^{\infty}\phi(u)^k\frac{(\gamma)_k\lambda^k}{k!}\frac{w^{\alpha\gamma-1}}{(w^\alpha+\lambda)^{\gamma+k}},\ \ (\text{using (\ref{wetlapmittag})})\\
	&=\frac{w^{\alpha\gamma-1}}{(w^\alpha+\lambda)^\gamma}\sum_{k=0}^{\infty}\frac{(-\gamma)_{(k)}}{k!}\left(-\frac{\phi(u)\lambda}{w^\alpha+\lambda}\right)^k=\frac{w^{\alpha\gamma-1}}{(w^\alpha+(1-\phi(u))\lambda)^\gamma},
\end{align*}
whose inversion yields
$
	\mathbb{E}e^{iu\mathcal{Y}_{\alpha,\gamma}(B)}=E_{\alpha,1}^\gamma(-(1-\phi(u))\lambda|B|^\alpha),\ u\in\mathbb{R}.$

\section{Concluding remarks}
The Poisson random field on $\mathbb{R}^2_+$ is a mathematical model that represent the pattern of points that are scattered in plane. In this paper, we have introduced and studied a fractional variant of the Poisson random field on $\mathbb{R}^2_+$. This variant is obtained by replacing the integer order derivatives by the Caputo fractional derivatives in the system of differential equations that governs the distribution of PRF. The main advantage of taking Caputo fractional derivative over the Riemann-Liouville fractional derivative is that it does not involve the fractional order derivative in the initial condition, which is more suitable for practical purposes. We presented a detailed study on the FPRF and highlighted the connection between FPRF and PRF.  In various studies, it has been observed the time-changed random processes and time fractional mathematical models provides a better models to study real life systems such as, in option pricing (see Carr and Wu (2004)), in biology (see Orovio \textit{et al.} (2014), Feng \textit{et al.} (2023)), \textit{etc}. The FPRF has potential applications in the fields of ecology and astronomy which require the data collected over a long period of time. Moreover, a generalized fractional Poisson process is introduced using a family of probability distributions. The Poisson and fractional Poisson probability distributions belong to this family of probability distributions.
As an application, we consider a time-changed linear and planer random motions driven by the GPP. In a particular case, the discussed planer random motion reduces to the planer random motion related to the fractional Poisson process.

Future studies may concern about the extension of this study on higher dimensional Euclidean space. It may also include the various characterizations for these processes.

\end{document}